\documentclass[10pt,a4paper]{article}
\usepackage[a4paper]{geometry}
\usepackage{amssymb,latexsym,amsmath,amsfonts,amsthm}
\usepackage{graphicx,comment}
\usepackage{epsfig}
\usepackage{tikz}
\usepackage{subcaption}
\usepackage[T1]{fontenc}
\newtheorem{theorem}{Theorem}[section]
\newtheorem{lemma}[theorem]{Lemma}
\newtheorem{problem}[theorem]{Problem}
\newtheorem{proposition}[theorem]{Proposition}

\theoremstyle{definition}
\newtheorem{definition}[theorem]{Definition}

\theoremstyle{remark}

\newtheorem{remark}[theorem]{Remark}

\numberwithin{equation}{section}

\title{Asymptotics of the minimum values of Riesz and logarithmic potentials generated by greedy energy sequences on the unit circle}

\date{\today}

\author{Abey L\'{o}pez-Garc\'{i}a \qquad Ryan E. McCleary}

\begin{document}

\maketitle

\begin{comment}
\renewcommand{\thefootnote}{\fnsymbol{footnote}}
\footnotetext[1]{Department of Mathematics, University of Central Florida, 4393 Andromeda Loop North, Orlando, FL 32816, USA. email: abey.lopez-garcia\symbol{'100}ucf.edu. Corresponding author.}
\footnotetext[2]{Department of Mathematics, University of Central Florida, 4393 Andromeda Loop North, Orlando, FL 32816, USA. email: remccleary\symbol{'100}knights.ucf.edu.}
\end{comment}

\begin{abstract} In this work we investigate greedy energy sequences on the unit circle for the logarithmic and Riesz potentials. By definition, if $(a_n)_{n=0}^{\infty}$ is a greedy $s$-energy sequence on the unit circle, the Riesz potential $U_{N,s}(x):=\sum_{k=0}^{N-1}|a_k-x|^{-s}$, $s>0$, generated by the first $N$ points of the sequence attains its minimum value at the point $a_{N}$, for every $N\geq 1$. In the case $s=0$ we minimize instead the logarithmic potential $U_{N,0}(x):=-\sum_{k=0}^{N-1}\log |a_{k}-x|$. We analyze the asymptotic properties of these extremal values $U_{N,s}(a_N)$, studying separately the cases $s=0$, $0<s<1$, $s=1$, and $s>1$. We obtain second-order asymptotic formulas for $U_{N,s}(a_N)$ in the cases $s=0$, $0<s<1$, and $s=1$ (the corresponding first-order formulas are well known). A first-order result for $s>1$ is proved, and it is shown that the normalized sequence $U_{N,s}(a_N)/N^s$ is bounded and divergent in this case. We also consider, briefly, greedy energy sequences in which the minimization condition is required starting from the point $a_{p+1}$ (instead of the point $a_{1}$ as previously stated), for some $p\geq 1$. For this more general class of greedy sequences, we prove a first-order asymptotic result for $0\leq s<1$.  
\smallskip

\textbf{Keywords:} Greedy energy sequence, Leja sequence, Riesz potential, Riemann zeta function.

\smallskip

\textbf{MSC 2020:} Primary 31A15, 31C20; Secondary 11M06.

\end{abstract}

\section{Introduction and main results}

\subsection{Background}

Let $K$ be a compact and infinite set in the complex plane. A sequence of points $\{a_{n}\}_{n=0}^{\infty}\subset K$ is called a \emph{Leja sequence} on $K$ if the following extremal property holds for all $n\geq 1$:
\[
\prod_{k=0}^{n-1}|a_{n}-a_{k}|=\max_{z\in K}\prod_{k=0}^{n-1}|z-a_{k}|.
\] 
In other words, if we define the polynomials $P_{n}(z):=\prod_{k=0}^{n-1}(z-a_{k})$, the condition is that the uniform norm $\|P_{n}\|_{K}$ is attained at the point $a_{n}$ for all $n\geq 1$. This old construction was introduced by Edrei in \cite{Edrei}, and reappeared later in the works of Leja \cite{Leja} and G\'{o}rski \cite{Gorski}, the latter in the context of the Newtonian potential in $\mathbb{R}^{3}$. It is very natural to do a comparative analysis of the asymptotic properties of Leja sequences and Fekete tuples on $K$. One observes that they share several important properties. Fekete tuples are defined as follows. If we let
\begin{equation}\label{eq:nthdiamK}
v_{n}(K):=\sup\Big\{\prod_{1\leq j<k\leq n} |x_{j}-x_{k}|: x_{1},\ldots,x_{n}\in K\Big\},\qquad n\geq 2,
\end{equation}
then a Fekete $n$-tuple for $K$ is a tuple $\omega_{n}=(x_{1,n},\ldots,x_{n,n})\in K^{n}$ for which the supremum in \eqref{eq:nthdiamK} is attained. The classical Fekete-Szeg\H{o} theorem asserts that the sequence of $n$-th diameters $\delta_{n}(K)=v_{n}(K)^{2/n(n-1)}$ converges to $c(K)$, the logarithmic capacity of $K$ (see e.g. \cite[Thm. 5.5.2]{Rans}). Interestingly, as it was already proved by Edrei in \cite{Edrei}, if $\{a_{n}\}_{n=0}^{\infty}$ is a Leja sequence on $K$, then the \emph{nested} configurations $\alpha_{n}=(a_{0},\ldots,a_{n-1})$ are asymptotically extremal in the sense that
\[
\lim_{n\rightarrow\infty}\prod_{0\leq i<j\leq n-1}|a_{i}-a_{j}|^{2/n(n-1)}=c(K).
\]   
Hence, this is a common property of Leja sequences and Fekete tuples. Another important property in common, closer to the subject of this paper, is that for the monic polynomials $P_{n}$ and $Q_{n}$ whose roots are the points in $\alpha_{n}$ and $\omega_{n}$ respectively, we have
\begin{equation}\label{eq:extrpoly}
\lim_{n\rightarrow\infty}\|P_{n}\|_{K}^{1/n}=\lim_{n\rightarrow\infty}\|Q_{n}\|_{K}^{1/n}=c(K),
\end{equation}
see \cite[Lemma 1]{Leja} and \cite[Thm. 5.5.4]{Rans}.

Let us consider the analogous constructions in the context of the Riesz potential. Given a fixed $s>0$, a sequence $\{a_{n}\}_{n=0}^{\infty}\subset K$ is called a \emph{greedy $s$-energy sequence} on $K$ if for every $n\geq 1$ we have
\[
\sum_{k=0}^{n-1} \frac{1}{|a_{n}-a_{k}|^s}=\inf_{z \in K} \sum_{k=0}^{n-1} \frac{1}{|z-a_{k}|^s}.
\]
The Riesz $s$-potential generated by the first $N$ points of the sequence $\{a_{n}\}_{n=0}^{\infty}$ will be denoted
\[
U_{N,s}(z):=\sum_{k=0}^{N-1}\frac{1}{|z-a_{k}|^s},\qquad z\in K.
\]
So we have
\begin{equation}\label{eq:minpotgreedy}
U_{N,s}(a_{N}) = \inf_{z \in K} U_{N,s}(z),\qquad N\geq 1.
\end{equation}
We will use the notation $\alpha_{N,s}:=(a_{0},\ldots,a_{N-1})$ for the tuple formed by the first $N$ points of the greedy $s$-energy sequence. We call this tuple the \textit{$N$-th section} of the sequence.

For a configuration $\omega=(x_1,\ldots,x_N)$ of $N \geq 2$ points in $K$, we define 
\[
E_s(\omega):=\sum_{1\leq i \neq j \leq N} \frac{1}{|x_i - x_j|^s} = 2 \sum_{1\leq i<j\leq N}\frac{1}{|x_i-x_j|^s},
\]
called the Riesz $s$-energy of $\omega$. We say that a configuration $\omega_{N,s}\in K^N$ is an $N$-point minimal $s$-energy configuration on $K$ if 
\begin{equation}\label{eq:minenerg}
E_{s}(\omega_{N,s})=\inf\{E_{s}(\omega): \omega\in K^N\}. 
\end{equation}
Let $\mathcal{P}(K)$ denote the space of all Borel probability measures supported on $K$. The Riesz $s$-energy of $\mu\in\mathcal{P}(K)$ is the double integral
\[
I_{s}(\mu):=\iint\frac{1}{|x-y|^{s}}\,d\mu(x)\,d\mu(y).
\]
If the planar compact set $K$ has positive Hausdorff dimension $\mathrm{dim}\,K$ and $0<s<\mathrm{dim}\,K$, the first-order asymptotic behavior of $E_{s}(\omega_{N,s})$ is the well-known formula
\begin{equation}\label{asympoptconf}
\lim_{N\rightarrow\infty}\frac{E_{s}(\omega_{N,s})}{N^{2}}=w_{s}(K)
\end{equation}
where the limit is the Wiener constant
\[
w_{s}(K)=\inf_{\mu\in\mathcal{P}(K)} I_{s}(\mu)<\infty
\]
see e.g. \cite[Thm. 4.4.9]{BorHarSaff}. Similarly to what happens in the Leja-Fekete context, in this range for $s$ formula \eqref{asympoptconf} is also valid if we replace $\omega_{N,s}$ by $\alpha_{N,s}$, see \cite[Thm. 2.1]{LopSaff}. Furthermore, by a result of Siciak (see \cite[Lemma 3.1]{Siciak}) the extremal values in \eqref{eq:minpotgreedy} satisfy
\begin{equation}\label{asympUnsgen}
\lim_{N\rightarrow\infty}\frac{U_{N,s}(a_{N})}{N}=w_{s}(K).
\end{equation}
In the range $s\geq\mathrm{dim}\,K$ the first-order asymptotic behavior of the quantities discussed is not understood in general, and additional assumptions on the set $K$ are usually imposed.  

In this work we focus on the particular case of the unit circle $K=S^{1}$, and analyze in detail the asymptotic behavior of the extremal values $U_{N,s}(a_{N})$ in \eqref{eq:minpotgreedy}, for all values of $s>0$. Our results  complement those obtained in \cite{LopWag}, which described in this context the asymptotic behavior of the energies $E_{s}(\alpha_{N,s})$. We also refine the asymptotic formula \eqref{eq:extrpoly} for the polynomials $P_{n}$.

It is a well known fact that for $0<s<1$, the normalized arc-length measure on $S^{1}$ is the unique measure in $\mathcal{P}(S^1)$ with minimal $s$-energy. We will denote this measure by $\sigma$. The $s$-energy of $\sigma$ is given explicitly by
\[
I_{s}(\sigma)=\frac{\Gamma(1-s)}{\Gamma\left(1-\frac{s}{2}\right)^{2}}=\frac{2^{-s}}{\sqrt{\pi}}\frac{\Gamma(\frac{1-s}{2})}{\Gamma(1-\frac{s}{2})},\qquad 0<s<1, 
\]
cf. \cite[Cor. A.11.4]{BorHarSaff}. On the other hand, in the range $s\geq 1$ every $\mu\in\mathcal{P}(S^{1})$ has infinite $s$-energy.

It was proved in \cite{Gotz} that for any $s>0$, every $N$-point minimal $s$-energy configuration $\omega_{N,s}$ on $S^{1}$ consists of $N$ equally spaced points, and in \cite{BrauHardSaff} a complete asymptotic expansion in $N$ of $E_{s}(\omega_{N,s})$ was found. The analyses in \cite{LopSaff, LopWag} revealed how the asymptotic behaviors of the quantities $E_{s}(\alpha_{N,s})$ and $E_{s}(\omega_{N,s})$ differ significantly when we look at first-order asymptotics in the range $s>\mathrm{dim}\,S^{1}=1$, and they also differ for second-order asymptotics in the range $0<s\leq 1$, in contrast to the similarities described above.

The method we have employed to obtain our asymptotic results (except Theorem~\ref{theo:genseq}) is based primarily on the representation of $U_{N,s}(a_{N})$ in terms of the binary expansion of $N$, which is our equation \eqref{eq:descUnan}. This type of formula was first used in \cite{LopMcc} in our study of greedy $\lambda$-energy sequences, associated with the nonsingular potentials $|x-y|^{\lambda}$, $\lambda>0$. The identity \eqref{eq:descUnan} for $U_{N,s}(a_{N})$ is deduced from a geometric description due to Bia\l as-Cie\.z and Calvi \cite{BiaCal} of the configurations $\alpha_{N,s}$ in terms of the binary expansion of $N$. Their result was stated for Leja sequences on the unit circle, which in this set have the same geometric structure as greedy $s$-energy sequences. 

In recent years there has been an increasing interest in the study of Leja and greedy energy sequences in Euclidean spaces, as they provide many interesting applications for different areas in mathematics. In addition to the references mentioned above, we find diverse applications in works on numerical analysis, numerical linear algebra, approximation theory \cite{BagCalRei,BMSV,CalManh,CalManh2,Chkifa,CorDrag,Marchi,JWZ,Reichel,TayTotik}, potential theory \cite{CorDrag,Gotz2,Lop,Pritsker,SaffTotik}, discrepancy theory \cite{Pau,St}, stochastic analysis \cite{NarJak}, among other areas.

\subsection{Statement of main results}

Let $(a_{n})_{n=0}^{\infty}$ be a Leja sequence on the unit circle, and let $U_{N,0}(z):=\sum_{k=0}^{N-1}\log\frac{1}{|z-a_{k}|}$, $N\geq 1$. The logarithmic capacity of the unit circle is $c(S^{1})=1$, hence the formula \eqref{eq:extrpoly} for $P_{N}$ is equivalent to 
\[
\lim_{N\rightarrow\infty}\frac{\log \|P_{N}\|_{S^{1}}}{N}=\lim_{N\rightarrow\infty}\frac{U_{N,0}(a_{N})}{N}=0.
\]
Our first result is the following improvement of this formula.

\begin{theorem}\label{theo:logcase}
Let $(a_{n})_{n=0}^{\infty}$ be a Leja sequence on the unit circle, and let $P_{N}(z)=\prod_{k=0}^{N-1}(z-a_{k})$. For all $N\geq 1$, we have
\begin{equation}\label{ineq:theo1}
0<\frac{\log \|P_{N}\|}{\log (N+1)}\leq 1,
\end{equation}
where $\|P_{N}\|=\sup_{z\in S^{1}}|P_{N}(z)|$. The upper bound is attained if and only if $N$ is of the form $N=2^{m}-1$, $m\geq 1$, and we have
\begin{equation}\label{liminfsupnorm}
\liminf_{N\rightarrow\infty}\frac{\log \|P_{N}\|}{\log(N+1)}=0.
\end{equation}
Moreover, for each $N\geq 1$ fixed, the sequence
\begin{equation}\label{eq:specsubs}
\left(\frac{\log\|P_{2^{k} N}\|}{\log(2^{k} N+1)}\right)_{k=0}^{\infty}
\end{equation} 
is decreasing and its limit is $0$. 
\end{theorem} 

We introduce now several definitions that play a central role in our analysis. For an integer $N\geq 1$, let $\tau_{b}(N)$ denote the number of $1$'s in the binary representation of $N$, that is, $\tau_{b}(N)=p$ if $N$ is of the form 
\begin{equation}\label{binrep}
N=2^{n_1}+2^{n_2}+\cdots+2^{n_{p}},\qquad n_{1}>n_{2}>\cdots>n_{p}\geq 0.
\end{equation}

\begin{definition}
Given an integer $p\geq 1$, let $\Theta_p$ denote the set of 
all vectors $\vec{\theta} = (\theta_1,\ldots,\theta_p)\in\mathbb{R}^{p}$ such that there exists an infinite sequence $\mathcal{N}\subset\mathbb{N}$ satisfying the following properties: $\tau_{b}(N)=p$ for all $N\in\mathcal{N}$, and for $N = 2^{n_{1}}+2^{n_{2}}+\cdots+ 2^{n_{p}}\in\mathcal{N}$ as in \eqref{binrep} we have
\begin{equation}\label{eq:thetadef:2}
\lim_{N\in\mathcal{N}}\frac{2^{n_k}}{N}=\theta_k,\qquad \mbox{for all}\,\,\,\,1\leq k\leq p. 
\end{equation}
\end{definition}

For example, we have $\vec{\theta}=(8/13, 4/13, 1/13, 0)\in\Theta_{4}$ since the sequence $2^{n+3}+2^{n+2}+2^{n}+1$, $n\geq 1$, generates in the limit that vector $\vec{\theta}$. The set $\Theta_{p}$ can be given a more direct definition, but the way it is defined above presents the asymptotic property \eqref{eq:thetadef:2}, which is precisely the one that is needed in our analysis. 

The reader can easily check that $\Theta_{p}$ consists of all vectors $\vec{\theta}=(\theta_{1},\ldots,\theta_{p})$ that can be written in the form
\begin{equation}\label{eq:altdefvectheta}
\vec{\theta}=\left(\frac{2^{n_{1}}}{M},\frac{2^{n_{2}}}{M},\ldots,\frac{2^{n_{t}}}{M},0,\ldots,0\right)
\end{equation}
where $M=2^{n_{1}}+2^{n_{2}}+\cdots+2^{n_{t}}$ is an odd integer with $n_{1}>n_{2}>\cdots>n_{t}=0$ and $1\leq t\leq p$. The number of zeros in \eqref{eq:altdefvectheta} is therefore $p-t$, if they appear.

\begin{definition}
We define the set $\Theta:=\bigcup_{p\in\mathbb{N}}\Theta_{p}$. Let $G:\Theta \times (0,\infty) \longrightarrow\mathbb{R}$ be the function given by
\begin{equation}\label{def:Gcapfunc}
G(\vec{\theta}; s):=\sum_{k=1}^{p} \theta_k^s\qquad \mbox{if}\,\,\,\,\vec{\theta}=(\theta_{1},\ldots,\theta_{p}).
\end{equation}
We also define the function $\Lambda:\Theta\longrightarrow\mathbb{R}$ by 
\begin{equation}\label{def:Lambdafunc}
\Lambda(\vec{\theta}) := \sum_{k=1}^{p} \theta_k\log\theta_k\qquad \mbox{if}\,\,\,\,\vec{\theta}=(\theta_{1},\ldots,\theta_{p}),
\end{equation}
where we understand that $\theta_k\log\theta_k=0$ if $\theta_k=0$. 
\end{definition}

We will show that for each fixed $s\in(0,\infty)$, the function $G(\vec{\theta};s)$ is uniformly bounded in $\vec{\theta}$, see \eqref{eq:gbounds} and \eqref{eq:boundsGsg1}. The function $\Lambda$ is also bounded, see Lemma~\ref{lembound}.  

Define the functions
\begin{align}
\overline{g}(s) & :=\sup_{\vec{\theta}\in\Theta} G(\vec{\theta}; s),\qquad 0<s<1,\label{def:littlegub}\\
\underline{g}(s) & :=\inf_{\vec{\theta}\in\Theta} G(\vec{\theta}; s),\qquad s>1,\label{def:littleglb}
\end{align}
and the constant
\begin{equation}\label{def:lclambda}
\lambda:=\inf_{\vec{\theta}\in\Theta} \Lambda(\vec{\theta}).
\end{equation}

In this paper, $\zeta(s)$ denotes as usual the classical Riemann zeta function $\sum_{n=1}^{\infty}n^{-s}$, $s>1$, or its analytic extension to $\mathbb{C}\setminus\{1\}$. Observe that this function takes negative values in the range $0<s<1$.

If $0 < s < 1$ and $(a_n)_{n=0}^{\infty}$ is a greedy $s$-energy sequence on $S^1$, then in virtue of \eqref{asympUnsgen} we have
\[
\lim_{N\to\infty}\frac{U_{N,s}(a_N)}{N}=I_{s}(\sigma). 
\]
This formula motivates the study of the second-order asymptotic behavior of $U_{N,s}(a_{N})$ in this range for $s$, which we describe now.

\begin{theorem}\label{theo:secondorderunanlimits}
Let $0 < s < 1$, and let $(a_n)_{n=0}^{\infty}$ be a greedy $s$-energy sequence on $S^1$. Then the sequence 
\begin{equation}\label{eq:secondorderunan}
\left(\frac{U_{N,s}(a_N) - N I_{s}(\sigma)}{N^s}\right)_{N=1}^{\infty} 
\end{equation}
is bounded and divergent, and we have
\begin{align}
\limsup_{N \to \infty} \frac{U_{N,s}(a_N) - N I_{s}(\sigma)}{N^s} & = (2^s-1)\frac{2\zeta(s)}{(2\pi)^s}, \label{eq:limsupunan}\\
\liminf_{N \to \infty} \frac{U_{N,s}(a_N) - N I_{s}(\sigma)}{N^s} & = \overline{g}(s)(2^s-1)\frac{2\zeta(s)}{(2\pi)^s}. \label{eq:liminfunan}
\end{align}
For each $\vec{\theta}\in\Theta$, the value $G(\vec{\theta};s)(2^s-1)\frac{2\zeta(s)}{(2\pi)^s}$ is a limit point of the sequence \eqref{eq:secondorderunan}.
\end{theorem}

By limit point we mean that there is a subsequence of \eqref{eq:secondorderunan} that converges to the indicated value.

The following result can be deduced from \cite[Thm. 2.10]{LopSaff}. However, in this particular context of the unit circle we wish to provide an alternative proof of \eqref{eq:firstorders1} using the ideas developed in the present work.

\begin{theorem}\label{theo:foase1}
If $(a_n)_{n=0}^{\infty}$ is a greedy $1$-energy sequence ($s=1$) on $S^1$, then
\begin{equation} \label{eq:firstorders1}
\lim_{N \to \infty} \frac{U_{N,1}(a_N)}{N\log N} = \frac{1}{\pi}.
\end{equation}
\end{theorem}

The following theorem describes the corresponding second-order asymptotic behavior of $U_{N,1}(a_N)$.  

\begin{theorem}\label{theo:soase1}
Let $(a_{n})_{n=0}^{\infty}$ be a greedy $s$-energy sequence on $S^1$ with $s=1$. Then the sequence
\begin{equation}\label{normsecordseqse1}
\left(\frac{U_{N,1}(a_N)-\frac{1}{\pi}N\log N}{N}\right)_{N=1}^{\infty}
\end{equation}
is bounded and divergent. Moreover, we have
\begin{align}
\limsup_{N\to\infty} \frac{U_{N,1}(a_N)-\frac{1}{\pi}N\log N}{N} & = \frac{1}{\pi}\left(\gamma + \log\left(\frac{8}{\pi}\right)\right),\label{eq:secordlimsups1}\\
\liminf_{N\to\infty} \frac{U_{N,1}(a_N)-\frac{1}{\pi}N\log N}{N} & = \frac{1}{\pi}\left(\gamma + \log\left(\frac{8}{\pi}\right)+\lambda\right),\label{eq:secordliminfs1}
\end{align}
where $\gamma=\lim_{N\to\infty}(\sum_{k=1}^{N}\frac{1}{k}-\log N)$ is the Euler-Mascheroni constant, and $\lambda$ is given by \eqref{def:lclambda}. For each $\vec{\theta}\in\Theta$, the quantity $\frac{1}{\pi}(\gamma + \log\left(\frac{8}{\pi}\right)+\Lambda(\vec{\theta}))$ is a limit point of the sequence \eqref{normsecordseqse1}.
\end{theorem}

Unlike in the cases $0<s<1$ and $s=1$, the first-order normalized sequence \eqref{seqUnsansg1} is divergent, as our next result demonstrates. 

\begin{theorem}\label{theo:foasg1}
Let $s>1$, and let $(a_{n})_{n=0}^{\infty}$ be a greedy $s$-energy sequence on $S^{1}$. Then the sequence
\begin{equation}\label{seqUnsansg1}
\left(\frac{U_{N,s}(a_{N})}{N^{s}}\right)_{N=1}^{\infty}
\end{equation}
is bounded and divergent, and we have
\begin{align}
\limsup_{N\rightarrow\infty}\frac{U_{N,s}(a_{N})}{N^{s}} & =(2^{s}-1)\frac{2\zeta(s)}{(2\pi)^{s}},\label{asymplimsupUsg1}\\
\liminf_{N\rightarrow\infty}\frac{U_{N,s}(a_{N})}{N^{s}} & =\underline{g}(s)(2^{s}-1)\frac{2\zeta(s)}{(2\pi)^{s}}.\label{asympliminfUsg1}
\end{align}
For each $\vec{\theta}\in\Theta$, the value $G(\vec{\theta};s)(2^{s}-1)\frac{2\zeta(s)}{(2\pi)^{s}}$ is a limit point of the sequence \eqref{seqUnsansg1}.
\end{theorem}

All previous results concern greedy energy sequences with a single initial point $a_{0}$ and all subsequent points $a_{n}$, $n\geq 1$, satisfying the optimality condition, i.e. minimizing the potential. It is also natural to consider greedy sequences in which the optimality condition is required to hold starting from the point $a_{p+1}$, for some fixed $p\geq 1$. Then the first $p+1$ points $a_{0},\ldots,a_{p}$ are viewed as given initial input.

Let 
\begin{equation}\label{def:kernelks}
k_{s}(z,w):=\begin{cases}
\log\frac{1}{|z-w|} & s=0,\\
\frac{1}{|z-w|^{s}} & s>0.
\end{cases}
\end{equation}
Let $p\geq 1$ and $a_{0},\ldots,a_{p}$ be $p+1$ distinct points on the unit circle. We shall say that $(a_{n})_{n=0}^{\infty}$ is a greedy $s$-energy sequence on $S^{1}$ with initial set $(a_{0},\ldots,a_{p})$ if 
\begin{equation}\label{eq:optgencase}
\sum_{k=0}^{n-1} k_{s}(a_{n},a_{k})=\inf_{z\in S^{1}}\sum_{k=0}^{n-1}k_{s}(z,a_{k}),\qquad n\geq p+1.
\end{equation}
In this setting, we shall use the same notations $\alpha_{N,s}=(a_{0},\ldots,a_{N-1})$, $U_{N,s}(z)=\sum_{k=0}^{N-1} k_{s}(z,a_{k})$, and $E_{s}(\alpha_{N,s})=\sum_{0\leq i\neq j\leq N-1}k_{s}(a_{i},a_{j})$ employed before. The method used to obtain our previous results is not applicable for this more general class of sequences since the connection with binary representations of natural numbers is lost. However, using standard potential theoretic arguments we can prove the following result in the range $0\leq s<1$ (the positive $s$-capacity range for $S^{1}$). 

\begin{theorem}\label{theo:genseq}
Let $0\leq s<1$, and let $(a_{n})_{n=0}^{\infty}$ be a greedy $s$-energy sequence on the unit circle with initial set of distinct points $(a_{0},\ldots,a_{p})$, $p\geq 1$. Then
\begin{equation}\label{eq:genseq}
\lim_{N\rightarrow\infty}\frac{E_{s}(\alpha_{N,s})}{N^{2}}=\lim_{N\rightarrow\infty}\frac{U_{N,s}(a_{N})}{N}=I_{s}(\sigma),
\end{equation}
where $I_{0}(\sigma)=0$. The sequence $(a_{n})_{n=0}^{\infty}$ is uniformly distributed, i.e., 
\begin{equation}\label{eq:unifdist}
\frac{1}{N}\sum_{k=0}^{N-1}\delta_{a_{k}}\longrightarrow \sigma,\qquad N\rightarrow\infty,
\end{equation}
in the weak-star topology. 
\end{theorem}

In the logarithmic case $s=0$, discrepancy properties of these sequences were studied by Steinerberger \cite{St}, proving in particular \eqref{eq:unifdist}. We finally mention some open problems for this more general class of sequences.

L\'{o}pez-Garc\'{i}a and Wagner \cite{LopWag} showed that for Leja sequences with a single initial point (i.e. $s=0$, $p=0$) we have
\[
0\leq \frac{E_{0}(\alpha_{N,0})+N\log N}{N}<\log(4/3),\qquad N\geq 2,
\]
which implies in particular $E_{0}(\alpha_{N})\thicksim -N\log N$. 

\begin{problem}
Let $(a_{n})_{n=0}^{\infty}$ be a Leja sequence on $S^{1}$ with more than one initial point (i.e. $s=0$, $p\geq 1$). Is it true that $E_{0}(\alpha_{N,0})=-N\log N+O(N)$? If $P_{N}(z)=\prod_{k=0}^{N-1}(z-a_{k})$ and $\|P_{N}\|$ is the $L^{\infty}$ norm on $S^{1}$, is it true that $\log\|P_{N}\|=O(\log N)$ and $\limsup_{N\rightarrow\infty}\log\|P_{N}\|/\log N=1$?  
\end{problem}

Let $s>0$ and let $(a_{n})_{n=0}^{\infty}$ be a greedy $s$-energy sequence with $p\geq 1$. We also expect in this setting that the sequences constructed as in \eqref{eq:secondorderunan} (when $0<s<1$), \eqref{seqUnsansg1} (when $s>1$), and \eqref{normsecordseqse1} (when $s=1$), are bounded and divergent, the latter property implying in particular that the limit \eqref{eq:firstorders1} is valid.  

\begin{problem}
Assume that $p\geq 1$. Describe in this situation the asymptotic behavior of the sequences \eqref{eq:secondorderunan} for $0<s<1$, \eqref{normsecordseqse1} for $s=1$, and \eqref{seqUnsansg1} for $s>1$.
\end{problem}

This paper is organized as follows. In Section~\ref{sec:logres} we prove Theorem~\ref{theo:logcase}. In Section~\ref{sec:defauxres} we introduce some important definitions and obtain several auxiliary results, including the key formula \eqref{eq:descUnan}. In Sections~\ref{sec:potcase}, \ref{sec:sequalone}, and \ref{sec:sgone}, we concentrate the asymptotic analysis of the cases $0<s<1$, $s=1$, and $s>1$ respectively, for sequences with a single initial point. In Section~\ref{sec:genseq} we prove Theorem~\ref{theo:genseq}.

\section{The logarithmic case: Proof of Theorem~\ref{theo:logcase}}\label{sec:logres}

Let $(a_{n})_{n=0}^{\infty}$ be a Leja sequence on the unit circle. Theorem~\ref{theo:logcase} will be a simple application of the following beautiful identity from \cite[Lemma 4]{CalManh}:
\begin{equation}\label{eq:CalManh}
\|P_{N}\|=\prod_{k=0}^{N-1}|a_{N}-a_{k}|=2^{\tau_{b}(N)},\qquad N\geq 1,
\end{equation}
where $\|P_{N}\|$ is the uniform norm on $S^{1}$ of the polynomial $P_{N}(z)=\prod_{k=0}^{N-1}(z-a_{k})$. 

\smallskip

\noindent\textit{Proof of Theorem}~\ref{theo:logcase}. For every integer $N\geq 1$, we have
\begin{equation}\label{keyineq}
N\geq 2^{\tau_{b}(N)}-1.
\end{equation}
Indeed, if $N$ is as indicated in \eqref{binrep}, then $n_{k}\geq p-k$ for each $1\leq k\leq p$, and so
\[
N=\sum_{k=1}^{p} 2^{n_{k}}\geq \sum_{k=1}^{p} 2^{p-k}=2^{p}-1=2^{\tau_{b}(N)}-1.
\]
Applying \eqref{eq:CalManh} and \eqref{keyineq} we get
\begin{equation}\label{eq:1}
\log \|P_{N}\|=\log (2^{\tau_{b}(N)})\leq \log(N+1).
\end{equation}
This justifies the second inequality in \eqref{ineq:theo1}. Clearly, equality in \eqref{eq:1} holds if and only if $N=2^{m}-1$, $m\geq 1$. It follows from \eqref{eq:CalManh} that $\log \|P_{N}\|=\tau_{b}(N) \log 2>0$. So \eqref{ineq:theo1} is justified. 

If we take the subsequence $N(r)=2^{r}$, then $\tau_{b}(N(r))=1$ for all $r$, hence
\[
\lim_{r\rightarrow\infty}\frac{\log \|P_{N(r)}\|}{\log (N(r)+1)}=\lim_{r\rightarrow\infty}\frac{\log 2}{\log(2^{r}+1)}=0
\]
so \eqref{liminfsupnorm} is justified.

For any $m\geq 1$, we have $\tau_{b}(m)=\tau_{b}(2m)$. Hence $\tau_{b}(2^{k} N)=\tau_{b}(N)$ for all $k\geq 0$, 
and so $\log \|P_{2^{k}N}\|=\log(2^{\tau_{b}(2^{k}N)})=\log(2^{\tau_{b}(N)})=\log\|P_{N}\|$ for all $k\geq 0$, i.e., the numerator in \eqref{eq:specsubs} is constant.\qed

\begin{figure}
\centering
\includegraphics[width=\linewidth]{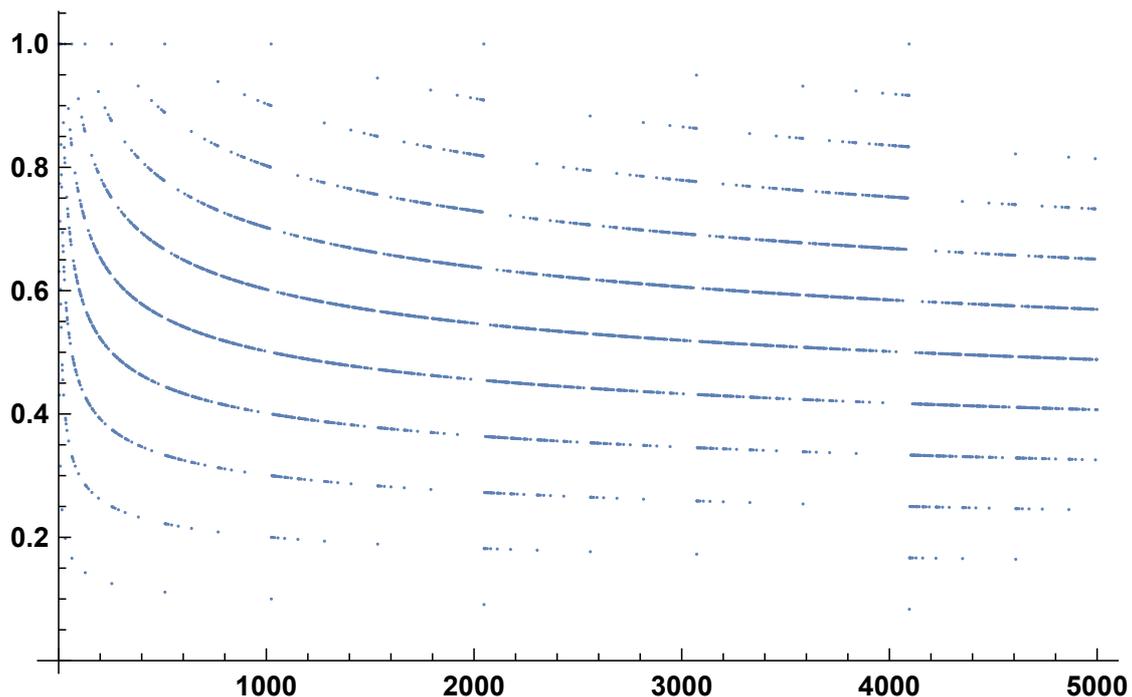}
\caption[]{Plot of the sequence $(\log\|P_{N}\|/\log(N+1))$ for $1 \leq N \leq 5000$.}
\label{fig1}
\end{figure}

\bigskip

Theorem~\ref{theo:logcase} is related to the following problem of P. Erd\H{o}s, which he formulated as Problem 21 in \cite{Erdos}. For a sequence $(z_{k})_{k=1}^{\infty}$ of complex numbers (not necessarily distinct) on the unit circle, let 
\[
A_{n}=A_{n}(z_{1},\ldots,z_{n}):=\max_{|z|=1}\prod_{k=1}^{n}|z-z_{k}|.
\] 
Is the sequence $(A_{n})_{n=1}^{\infty}$ unbounded for any such sequence $(z_{k})_{k=1}^{\infty}$? Erd\H{o}s conjectured that the answer is affirmative, and also asked to estimate how fast the sequence $\max_{1\leq n\leq N} A_{n}$ diverges to infinity. Wagner \cite{Wagner} managed to prove the conjecture, and gave a quantitative estimate of the rate of growth of $\max_{1\leq n\leq N}A_{n}$. His estimate was improved by Beck \cite{Beck}, who showed that 
\[
\max_{1\leq n\leq N} A_{n}>N^{c}
\]    
where $c>0$ is an absolute constant. 

Let $\{x_{k}\}_{k=1}^{\infty}\subset[0,1)$ be the van der Corput sequence, that is, the sequence defined inductively by the formula
\[
x_{2^{k}+l}=2^{-k-1}+x_{l},\qquad 1\leq l\leq 2^{k},\quad k\geq 1,
\]
starting from $x_{1}=0$, $x_{2}=1/2$. Erd\H{o}s also observed (see \cite[Problem 4.1]{HayLin}) that if one maps this sequence to the unit circle to obtain the sequence $z_{k}=e^{2\pi i x_{k}}$ (particular example of Leja sequence on $S^{1}$), then one has $A_{n}\leq n+1$ with equality if and only if $n=2^{k}-1$ for some integer $k$. Erd\H{o}s asked whether it was possible to construct other sequences $(z_{k})_{k=1}^{\infty}$ for which $A_{n}$ can increase more slowly than the indicated linear estimate. Linden \cite{Linden} proved that this is possible, constructing a sequence $(z_{k})_{k=1}^{\infty}$ on the unit circle for which
\[
\limsup_{n\rightarrow\infty}\frac{\log A_{n}}{\log n}<1.
\]    

\section{Some definitions and auxiliary results}\label{sec:defauxres}

In this section we prove some identities that are needed for our asymptotic analysis later. The most crucial is the representation of $U_{N,s}(a_{N})$ in terms of the binary expansion of $N$.

\begin{proposition} \label{prop:sdinequality}
Let $0<s<1$, and let $(a_n)_{n=0}^{\infty}$ be a greedy $s$-energy sequence on $S^1$. We have $U_{N,s}(a_N)<N I_{s}(\sigma)$ for all $N\geq 1$.
\end{proposition}
\begin{proof}
We have
\[
U_{N,s}(a_N)=\int U_{N,s}(a_N)\,d\sigma(x)<\int U_{N,s}(x)\,d\sigma(x) =\sum_{k=0}^{N-1}\int \frac{1}{|a_k - x|^{s}}\,d\sigma(x) = N I_{s}(\sigma),
\]
where the final equality follows from the fact that, by symmetry, $\int |a_k - x|^{-s}\,d\sigma(x) = I_{s}(\sigma)$.
\end{proof}

The optimal value \eqref{eq:minenerg} for $K=S^{1}$ will be denoted by
\[
\mathcal{L}_{s}(N):=\inf\{E_{s}(\omega): \omega\in (S^{1})^{N}\},\qquad N\geq 2.
\]
Recall that this value is the $s$-energy of the collection of $N$-th roots of unity, which easily implies the identity 
\[
\mathcal{L}_{s}(N)=2^{-s} N \sum_{k=1}^{N-1} \left(\sin\frac{k\pi}{N}\right)^{-s}.
\]
We also set $\mathcal{L}_{s}(1):=0$.  

In this paper, an important role will be played by the quantities
\[
\mathcal{U}_{s}(N) := \sum_{k=1}^{N} \left|\exp(\pi i/N) - \exp(2\pi k i/N)\right|^{-s}.
\]
This is the Riesz potential generated by the $N$-th roots of unity $e_{k}=\exp(2\pi k i/N)$, $1\leq k\leq N$, evaluated at the midpoint of the short arc between $e_{N}$ and $e_{1}$.

\begin{proposition} \label{prop:lformula}
If $N \geq 2$ and $\{e_1,\ldots,e_{N}\}$ are the $N$-th roots of unity, then
\[
\sum_{k=2}^{N} \frac{1}{|e_k - e_1|^s} = \frac{\mathcal{L}_{s}(N)}{N}.
\]
\end{proposition} 
\begin{proof}
Let $S=\{e_1,\ldots,e_{N}\}$. We have
\[
\mathcal{L}_{s}(N) = \sum_{k=1}^{N}\sum_{z \in S \setminus \{e_k\}} \frac{1}{|z - e_k|^{s}}.
\]
By symmetry, this implies
\[
\mathcal{L}_{s}(N) = N \sum_{k=2}^{N} \frac{1}{|e_k - e_1|^{s}},
\]
and dividing by $N$ yields the desired equality. 
\end{proof}

\begin{proposition} \label{prop:uInTermsOfL}
If $N \geq 1$, then for any $s > 0$, 
\begin{equation}\label{calUsNLsN}
\mathcal{U}_{s}(N) = \frac{\mathcal{L}_{s}(2N)}{2N} - \frac{\mathcal{L}_{s}(N)}{N}.
\end{equation}
\end{proposition}
\begin{proof}
Let $e_k = \exp(2\pi k i / 2N)$ for $1 \leq k \leq 2N$. By the definition of $\mathcal{U}_s(N)$, we 
evidently have
\begin{align*}
\mathcal{U}_{s}(N) = \sum_{k=1}^{N} \frac{1}{|e_{2k} - e_1|^{s}}
= \sum_{k=2}^{2N} \frac{1}{|e_k - e_1|^{s}} - \sum_{k=1}^{N-1}\frac{1}{|e_{2k+1} - e_1|^{s}}.
\end{align*}
By Proposition~\ref{prop:lformula}, this implies
\[
\mathcal{U}_{s}(N) = \frac{\mathcal{L}_{s}(2N)}{2N} - \frac{\mathcal{L}_{s}(N)}{N}.
\]
\end{proof}

If $(a_n)_{n=0}^{\infty}$ is a greedy $s$-energy sequence on $S^1$, it is obvious that the values  $U_{n,s}(a_n)$ will not be changed if we perform a rotation of the sequence. Hence, we may assume without loss of generality that the initial point of a greedy sequence is always the point $a_0=1$. In Proposition~\ref{prop:geodesc} below, we tacitly assume this condition. 

For an $M$-tuple $A = (a_1,a_2,\ldots,a_M)$ and an $N$-tuple $B = (b_1,b_2,\ldots,b_N)$, we define the concatenation $(A,B) := (a_1,a_2,\ldots,a_M,b_1,b_2,\ldots,b_N)$. In \cite{BiaCal}, Proposition~\ref{prop:geodesc} was proved for Leja sequences. But as it was remarked in \cite{LopWag}, this result holds for greedy $s$-energy sequences as well, for any value of $s>0$, since greedy $s$-energy sequences coincide geometrically with Leja sequences on $S^1$. We now state this result in terms of greedy $s$-energy sequences on $S^1$.

\begin{proposition} \label{prop:geodesc}
Let $s > 0$ and $(a_n)_{n=0}^{\infty}$ be a greedy $s$-energy sequence on $S^1$. Then
\begin{enumerate}
\item[$1)$] Any $2^n$-th section of $(a_n)$ is formed by the $2^n$-th roots of unity.
\item[$2)$] For any $n \geq 0$, let $\alpha_{2^{n+1}}$ be the $2^{n+1}$-th section of $(a_n)$. Then $\alpha_{2^{n+1}}$ 
contains $\alpha_{2^n}$ as its first $2^n$ points, and there exists a $2^n$-th root $\rho$ of $-1$ and a greedy $s$-energy sequence $({a}'_{n})_{n=0}^{\infty}$ on $S^1$ such that $\alpha_{2^{n+1}} = (\alpha_{2^n}, \rho {\alpha}'_{2^n})$. (Here, ${\alpha}'_{N}$ denotes the 
$N$-th section of $({a}'_{n})$.)
\item[$3)$] Iterating $2)$, we see that for any $N = 2^{n_1} + 2^{n_2} + \cdots + 2^{n_p}$, $n_1 > n_2 > \cdots > n_p \geq 0$, 
there exists for each integer $k$ with $1 \leq k \leq p$ a greedy $s$-energy sequence $(a_n^k)_{n=0}^{\infty} \subset S^1$ (with $a_0^k = 1$) such that 
\[
\alpha_{N} = (\alpha_{2^{n_1}}^{1}, \rho_1\alpha_{2^{n_2}}^{2}, \rho_1\rho_2\,\alpha_{2^{n_3}}^{3},\ldots,\left(\prod_{i=1}^{p-1} \rho_i\right)\alpha_{2^{n_p}}^{p})
\]
for some numbers $\rho_i$ that are $2^{n_i}$-th roots of $-1$. Following our usual convention, we have denoted the $N$-th section
of $(a_n^k)$ by $\alpha_N^k$. 
\end{enumerate}
\end{proposition}

We now express $U_{N,s}(a_N)$ in terms of the binary expansion of $N$. All our asymptotic analysis will be based on the identity \eqref{eq:descUnan}. 

\begin{proposition} \label{prop:descUnan}
Let $s>0$ be arbitrary, and let $(a_{n})_{n=0}^{\infty}$ be a greedy $s$-energy sequence on $S^{1}$. If $N = 2^{n_1} + \cdots + 2^{n_p}$, $n_1 > \cdots > n_p \geq 0$,  then 
\begin{equation} \label{eq:descUnan}
U_{N,s}(a_{N}) = \sum_{k=1}^{p}\mathcal{U}_{s}(2^{n_{k}}).
\end{equation}
\end{proposition}
\begin{proof}
The proof is by induction on $p$. For simplicity of notation, we write $U_{N}$ instead of $U_{N,s}$ and $\alpha_{N}$ instead of $\alpha_{N,s}$. If $p=1$ and $N=2^{n}$, then by Proposition~\ref{prop:geodesc}, the points in $\alpha_{N} = (a_{0},\ldots,a_{N-1})$ are the $N$-th roots of unity. The point $a_{N}$ is the midpoint of one of the arcs between adjacent points in $\alpha_{N}$, hence $U_{N}(a_{N}) = \mathcal{U}_{s}(N)$.

For a fixed $p\geq 1$, assume as induction hypothesis that \eqref{eq:descUnan} is valid for every greedy $s$-energy sequence on $S^1$ and every $N\geq 1$ with binary representation of length $p$. Let $\widetilde{N}=2^{n_{1}}+\cdots+2^{n_{p}}+2^{n_{p+1}}$, $n_{1}>\cdots>n_{p}>n_{p+1}\geq 0$. Then, by Proposition~\ref{prop:geodesc}, we can write
\[
\alpha_{\widetilde{N}}=(\alpha_{2^{n_{1}}},\rho\,{\alpha}'_{\widetilde{N}-2^{n_1}}),
\]
where $\rho$ satisfies $\rho^{2^{n_{1}}}=-1$, and ${\alpha}'_{\widetilde{N}-2^{n_1}}$ is the section of order $\widetilde{N}-2^{n_{1}}$ of a greedy $s$-energy sequence $({a}'_{k})_{k=0}^{\infty}$. So $a_{k}=\rho\,{a}'_{k-2^{n_1}}$, $2^{n_{1}}\leq k\leq \widetilde{N}$, where $({a}'_{0},\ldots,{a}'_{\widetilde{N}-2^{n_1}})={\alpha}'_{\widetilde{N}-2^{n_1}+1}$. We have
\[
U_{\widetilde{N}}(a_{\widetilde{N}})=U_{2^{n_1}}(a_{\widetilde{N}})+\sum_{k=2^{n_1}}^{\widetilde{N}-1}\frac{1}{|a_{k}-a_{\widetilde{N}}|^{s}}.
\]
The point $a_{\widetilde{N}}$ is the midpoint of one of the arcs between two adjacent points in $\alpha_{2^{n_1}}$, hence 
\[
U_{2^{n_1}}(a_{\widetilde{N}})=\mathcal{U}_{s}(2^{n_1}).
\] 
We also have
\[
\sum_{k=2^{n_1}}^{\widetilde{N}-1}\frac{1}{|a_{k}-a_{\widetilde{N}}|^{s}}=
\sum_{k=2^{n_1}}^{\widetilde{N}-1}\frac{1}{|\rho\, {a}'_{k-2^{n_1}}-\rho\, {a}'_{\widetilde{N}-2^{n_1}}|^{s}}=\sum_{k=0}^{\widetilde{N}-2^{n_1}-1}\frac{1}{|{a}'_{k}-{a}'_{\widetilde{N}-2^{n_1}}|^{s}}=U_{M}({a}'_{M}),
\]
where $M:=\widetilde{N}-2^{n_{1}}$. The number $M=2^{n_{2}}+\cdots+2^{n_{p+1}}$ has a binary representation of length $p$, so by induction hypothesis we obtain
\[
U_{M}(a_{M}')=\sum_{k=2}^{p+1}\mathcal{U}_{s}(2^{n_{k}}).
\]
In conclusion, $U_{\widetilde{N}}(a_{\widetilde{N}})=\sum_{k=1}^{p+1}\mathcal{U}_{s}(2^{n_k})$, which finishes the proof of \eqref{eq:descUnan}.
\end{proof} 

\begin{remark} For a greedy $s$-energy sequence $(a_{n})_{n=0}^{\infty}\subset S^{1}$, the energy $E_{s}(\alpha_{N,s})$ of the $N$-th section and the optimal values of the Riesz potential \eqref{eq:minpotgreedy} are linked by the simple relation
\[
E_{s}(\alpha_{N,s})=2\sum_{j=1}^{N-1}\sum_{i=0}^{j-1}\frac{1}{|a_{i}-a_{j}|^{s}}
=2\sum_{j=1}^{N-1}U_{j,s}(a_{j}).
\]
\end{remark}

\section{The case $0<s<1$}\label{sec:potcase}

In this section we will use the following notations. For $0<s<1$ and $N \geq 1$, let
\begin{align}
\mathcal{R}_{s}(N) & :=\frac{\mathcal{L}_{s}(N) - N^2I_{s}(\sigma)}{N^{1+s}},\label{def:RsN}\\
\mathcal{W}_{s}(N) & := \frac{\mathcal{U}_{s}(N)-NI_{s}(\sigma)}{N^s}.\label{def:WsN}
\end{align}
The following formula is a particular application of \cite[Eq. (1.18)]{BrauHardSaff} (take $p=0$ in the referenced equation): 
\begin{equation}\label{eq:asympRsN}
\lim_{N \to \infty} \mathcal{R}_{s}(N)=\frac{2\zeta(s)}{(2\pi)^s},\qquad 0<s<1.
\end{equation} 

We deduce the following result for \eqref{def:WsN}. 
\begin{proposition} \label{prop:secondorderUcal}
We have
\begin{equation}\label{asympcalWsNsl1}
\lim_{N\rightarrow\infty}\mathcal{W}_{s}(N)=
(2^{s} - 1)\frac{2\zeta(s)}{(2\pi)^s},\qquad 0<s<1.
\end{equation}
\end{proposition}
\begin{proof}
We have
\begin{align*}
\frac{\mathcal{U}_{s}(N) - N I_{s}(\sigma)}{N^{s}} &= \frac{\mathcal{L}_s(2N) - 2\mathcal{L}_{s}(N) - 2N^2 I_{s}(\sigma)}{2N^{1+s}} \\
&= 2^{s}\frac{\mathcal{L}_{s}(2N) - (2N)^2I_{s}(\sigma)}{(2N)^{1+s}}-\frac{\mathcal{L}_{s}(N)-N^2 I_{s}(\sigma)}{N^{1+s}} \\
&= 2^{s}\mathcal{R}_{s}(2N) - \mathcal{R}_{s}(N)
\end{align*}
and the result follows from \eqref{eq:asympRsN}.
\end{proof}
Note that the limit value in \eqref{asympcalWsNsl1} is negative. 

We establish now bounds for the function $G(\vec{\theta};s)$ defined in \eqref{def:Gcapfunc}, in the range $0<s<1$.

\begin{lemma}\label{lem:boundg}
For all $\vec{\theta}\in\Theta$ we have
\begin{equation}\label{eq:gbounds}
1 \leq G(\vec{\theta}; s) < \frac{2^s}{2^s-1},\qquad 0<s<1.
\end{equation}
\end{lemma}
\begin{proof}
Fix $\vec{\theta}=(\theta_{1},\ldots,\theta_{p})\in\Theta$. Since $\sum_{k=1}^{p} \theta_k = 1$ and the function $x\mapsto x^s$ is concave, we obtain
\begin{align*}
1 = (\theta_1+\cdots+\theta_p)^s \leq \theta_1^s+\cdots+\theta_p^s = G(\vec{\theta}; s).
\end{align*}

Recall that vectors in $\Theta_p$ can be described as follows: For each $\vec{\theta}\in\Theta_p$, there is an odd integer $M= 2^{n_1}+\cdots+2^{n_t}$, where $t \leq p$ and $n_1 > n_{2}>\cdots > n_t=0$ such that
\begin{align*}
\theta_k & =\frac{2^{n_k}}{M},\qquad 1\leq k\leq t,\\
\theta_{k} & =0,\qquad t+1\leq k\leq p.
\end{align*} 
Using this property, we obtain
\begin{align*}
G(\vec{\theta};s) &= \sum_{k=1}^{p} \theta_k^s = \sum_{k=1}^{t}\left(\frac{2^{n_k}}{M}\right)^s 
\leq \sum_{k=1}^{t} 2^{s(n_k-n_1)} \\
&< \sum_{j=0}^{\infty} (2^s)^{-j}=\frac{2^s}{2^s-1}.
\end{align*}
\end{proof}

Recall the definition of the function $\overline{g}(s)$ in \eqref{def:littlegub}. By Lemma~\ref{lem:boundg}, we have 
\[
1<\overline{g}(s)\leq \frac{2^{s}}{2^s-1},\qquad 0<s<1.
\] 
In particular, the limit values in \eqref{eq:limsupunan} and \eqref{eq:liminfunan} are different. We are now ready for the proof of Theorem~\ref{theo:secondorderunanlimits}.

\noindent\textit{Proof of Theorem}~\ref{theo:secondorderunanlimits}. First, note that if $N=2^{n_1}+\cdots+2^{n_p}$, then
\begin{align*}
\frac{U_{N,s}(a_N)-N I_{s}(\sigma)}{N^s}
&= \sum_{k=1}^{p}\frac{\mathcal{U}_{s}(2^{n_k})-2^{n_k}I_{s}(\sigma)}{N^s}\\
&= \sum_{k=1}^{p}\frac{\mathcal{U}_{s}(2^{n_k})-2^{n_k}I_{s}(\sigma)}{(2^{n_k})^s}\left(\frac{2^{n_k}}{N}\right)^s\\
&= \sum_{k=1}^{p}\mathcal{W}_{s}(2^{n_k})\left(\frac{2^{n_k}}{N}\right)^s.
\end{align*}

By \eqref{asympcalWsNsl1}, the sequence $(\mathcal{W}_{s}(N))_{N=1}^{\infty}$ is bounded. Let $C>0$ be an upper bound for all $|\mathcal{W}_{s}(N)|$. Then, for $N=2^{n_{1}}+\cdots+2^{n_{p}}$ we have
\begin{align*}
\left|\frac{U_{N,s}(a_N) - N I_{s}(\sigma)}{N^s}\right|
 & \leq \sum_{k=1}^{p}\left|\mathcal{W}_{s}(2^{n_k})\right|\left(\frac{2^{n_k}}{N}\right)^s \leq C \sum_{k=1}^{p} 2^{(n_k-n_1)s}\\ & < C \sum_{k=0}^{\infty} (2^s)^{-k} = C\frac{2^s}{2^s - 1}. 
\end{align*}
This shows that the sequence \eqref{eq:secondorderunan} is bounded.

Fix $\vec{\theta}=(\theta_1,\ldots,\theta_p)\in\Theta_p$, for some $p\in\mathbb{N}$. 
By definition, there is an infinite sequence $\mathcal{N}$ of integers $N=2^{n_1}+\cdots+2^{n_p}$ with binary representation of length $p$ such that
\begin{equation}\label{seqassoctheta}
\lim_{N\in\mathcal{N}}\frac{2^{n_k}}{N} = \theta_k,\qquad 1\leq k\leq p.
\end{equation}
We then have
\begin{align}
\lim_{N\in\mathcal{N}}\frac{U_{N,s}(a_N)-NI_{s}(\sigma)}{N^s} 
&= \lim_{N\in\mathcal{N}}\sum_{k=1}^{p}\mathcal{W}_{s}(2^{n_k})\left(\frac{2^{n_k}}{N}\right)^{s} \notag\\
&=G(\vec{\theta}; s)(2^s-1)\frac{2\zeta(s)}{(2\pi)^s}.\label{limspesub}
\end{align}
From this we deduce several consequences. First, $G(\vec{\theta}; s)(2^s-1)\frac{2\zeta(s)}{(2\pi)^s}$ is a limit point of the sequence \eqref{eq:secondorderunan}. Second, since $G(\vec{\theta}; s)\geq1$, for such sequences $\mathcal{N}$ satisfying \eqref{seqassoctheta} we have
\[
\lim_{N\in\mathcal{N}}\frac{U_{N,s}(a_{N})-NI_{s}(\sigma)}{N^s} \leq (2^s-1)\frac{2\zeta(s)}{(2\pi)^s}.
\]
(Recall that $\zeta(s)<0$ for $0<s<1$.) Third, since $p$ and $\vec{\theta}$ were arbitrary, \eqref{limspesub} implies
\begin{equation}\label{halfineqliminf}
\liminf_{N\to\infty}\frac{U_{N,s}(a_N)-NI_{s}(\sigma)}{N^s}\leq \overline{g}(s)(2^s-1)\frac{2\zeta(s)}{(2\pi)^s}.
\end{equation}

We start now the proof of \eqref{eq:limsupunan}. Since
\[
\lim_{n\to\infty} \frac{U_{2^n,s}(a_{2^n})-2^n I_{s}(\sigma)}{(2^n)^s}=\lim_{n\to\infty}\mathcal{W}_{s}(2^n)= (2^s - 1)\frac{2\zeta(s)}{(2\pi)^s},
\]
we have
\begin{equation}
\limsup_{N\to\infty} \frac{U_{N,s}(a_N) - N I_{s}(\sigma)}{N^s} \geq (2^s - 1)\frac{2\zeta(s)}{(2\pi)^s}.\label{eq:limsupgeq}
\end{equation}
Let us prove the reverse inequality. 

Suppose that $\mathcal{N}\subset\mathbb{N}$ is a sequence such that
\begin{equation}\label{eq:convsubseq}
\left(\frac{U_{N,s}(a_{N})-NI_{s}(\sigma)}{N^s}\right)_{N\in\mathcal{N}}\quad\text{converges,}
\end{equation}
and let us show that its limit is $\leq (2^{s}-1)\frac{2\zeta(s)}{(2\pi)^{s}}$. First, suppose there exists $p\in\mathbb{N}$ such that  $\tau_{b}(N) = p$ for infinitely many integers $N\in\mathcal{N}$. If necessary, we may choose a subsequence $\widehat{\mathcal{N}}\subset\mathcal{N}$ such that 
\[
\lim_{N\in\widehat{\mathcal{N}}} \frac{2^{n_k}}{N}=\theta_k,\qquad 1\leq k\leq p,\quad N=2^{n_{1}}+\cdots+2^{n_{p}},
\]
for some $\vec{\theta}=(\theta_1,\ldots,\theta_p)\in\Theta_p$. 
In this case, we have
\[
\lim_{N\in\mathcal{N}} \frac{U_{N,s}(a_{N})-NI_{s}(\sigma)}{N^s}= \lim_{N\in\widehat{\mathcal{N}}} \frac{U_{N,s}(a_{N})-NI_{s}(\sigma)}{N^s}\leq (2^s-1)\frac{2\zeta(s)}{(2\pi)^s}.
\]
The final inequality above follows from the fact that $\widehat{\mathcal{N}}$ is a subsequence corresponding to some $\vec{\theta}\in\Theta_p$. We have already shown this inequality to be true for such subsequences. 

Now, assume that $\tau_{b}(N) \to \infty$ as $N \to \infty$ along the sequence $\mathcal{N}$. Fix $0<\epsilon<1$. Let $M$ be an integer such that
\begin{equation}\label{smalltail}
\sum_{k=M}^{\infty} 2^{-sk} < \epsilon. 
\end{equation}
Now, for $N=2^{n_{1}}+2^{n_{2}}+\cdots+2^{n_{\tau_{b}(N)}}\in\mathcal{N}$ we write
\[
\frac{U_{N,s}(a_{N})-NI_{s}(\sigma)}{N^s} = S_{N,1} + S_{N,2}
\]
where
\begin{align}
S_{N,1} & := \sum_{k=1}^{M} \mathcal{W}_{s}(2^{n_k})\left(\frac{2^{n_k}}{N}\right)^{s} \label{def:SN1sl1}\\
S_{N,2} & := \sum_{k=M+1}^{\tau_{b}(N)} \mathcal{W}_{s}(2^{n_k})\left(\frac{2^{n_k}}{N}\right)^{s}.\label{def:SN2sl1}
\end{align}
If $M \geq \tau_{b}(N)$, then $S_{N,2}$ is understood to be zero. With $C$ again the upper bound for all $|\mathcal{W}_{s}(N)|$,  we obtain
\begin{align}
|S_{N,2}| &\leq C\sum_{k=M+1}^{\tau_{b}(N)} \left(\frac{2^{n_k}}{N}\right)^s
\leq C\sum_{k=M+1}^{\tau_{b}(N)} 2^{-s(n_1-n_k)} \notag\\
&\leq C\sum_{k=M+1}^{\tau_{b}(N)} 2^{-s(k-1)}
< C\sum_{k=M}^{\infty} 2^{-sk} < C\epsilon. \label{est:SN2}
\end{align}

Let $\widetilde{N}:=2^{n_{1}}+2^{n_{2}}+\cdots+2^{n_{M}}$. Observe that this quantity varies with $N$. Then
\begin{align*}
S_{N,1} = \sum_{k=1}^{M} \mathcal{W}_{s}(2^{n_k})\left(\frac{2^{n_k}}{\widetilde{N}}\right)^{s}\left(\frac{\widetilde{N}}{N}\right)^{s}=\left(\frac{\widetilde{N}}{N}\right)^{s}\sum_{k=1}^{M} \mathcal{W}_{s}(2^{n_k})\left(\frac{2^{n_k}}{\widetilde{N}}\right)^{s}.
\end{align*}
Now, note that
\begin{align*}
\frac{\widetilde{N}}{N} = \frac{2^{n_1}+\cdots+2^{n_M}}{N}
= 1 - \frac{2^{n_{M+1}}+\cdots+2^{n_{\tau_{b}(N)}}}{N}.
\end{align*}
Since
\[
\frac{2^{n_k}}{N} \leq 2^{-(k-1)},\qquad1\leq k\leq\tau_{b}(N), 
\]
we have
\begin{align*}
\frac{\widetilde{N}}{N} > 1 - \sum_{k=M+1}^{\infty} 2^{-(k-1)} > 1 - \epsilon,
\end{align*}
which implies
\[
\left(\frac{\widetilde{N}}{N}\right)^{s} > (1 - \epsilon)^{s}.
\]
For each $1\leq k\leq M$, the sequence $2^{n_k}/\widetilde{N}$ is bounded, so we may choose a subsequence
$\widehat{\mathcal{N}}\subset \mathcal{N}$ such that
\begin{equation}\label{def:subseqNhat}
\lim_{N\in\widehat{\mathcal{N}}}\frac{2^{n_k}}{\widetilde{N}} = \theta_k,\qquad 1\leq k\leq M,
\end{equation}
for some $\vec{\theta}=(\theta_{1},\ldots,\theta_{M})\in\Theta_M$. 

Observe that by Proposition~\ref{prop:sdinequality},
\[
\mathcal{W}_{s}(2^{n})=\frac{\mathcal{U}_{s}(2^{n})-2^{n}I_{s}(\sigma)}{2^{ns}}=\frac{U_{n,s}(a_{2^{n}})-2^{n}I_{s}(\sigma)}{2^{ns}}<0.
\]
Hence
\[
S_{N,1}=\left(\frac{\widetilde{N}}{N}\right)^{s}\sum_{k=1}^{M}\mathcal{W}_{s}(2^{n_k})\left(\frac{2^{n_k}}{\widetilde{N}}\right)^{s}<(1-\epsilon)^{s}\sum_{k=1}^{M}\mathcal{W}_{s}(2^{n_k})\left(\frac{2^{n_k}}{\widetilde{N}}\right)^{s},
\]
and we obtain
\begin{align*}
\limsup_{N\in\widehat{\mathcal{N}}} S_{N,1} &\leq \limsup_{N\in\widehat{\mathcal{N}}}\,\, (1-\epsilon)^{s}\sum_{k=1}^{M}\mathcal{W}_{s}(2^{n_k})\left(\frac{2^{n_k}}{\widetilde{N}}\right)^{s}\\
& = (1-\epsilon)^{s}G(\vec{\theta};s)(2^{s}-1)\frac{2\zeta(s)}{(2\pi)^s}\\
&\leq (1-\epsilon)^{s}(2^{s}-1)\frac{2\zeta(s)}{(2\pi)^s}.
\end{align*}
It now follows that
\begin{align*}
\lim_{N\in\widehat{\mathcal{N}}} \frac{U_{N,s}(a_N)-NI_{s}(\sigma)}{N^s} &= \lim_{N\in\widehat{\mathcal{N}}}\left(S_{N,1}+S_{N,2}\right) 
\leq \limsup_{N\in\widehat{\mathcal{N}}} S_{N,1} + \limsup_{N\in\widehat{\mathcal{N}}} S_{N,2}\\
&\leq (1-\epsilon)^{s}(2^{s}-1)\frac{2\zeta(s)}{(2\pi)^s} + C\epsilon. 
\end{align*}
Letting $\epsilon\rightarrow 0$ we get
\[
\lim_{N\in\mathcal{N}} \frac{U_{N,s}(a_N)-NI_{s}(\sigma)}{N^s} \leq (2^s-1)\frac{2\zeta(s)}{(2\pi)^s}.
\]
Hence, we have shown that
\begin{equation}
\limsup_{N\to\infty} \frac{U_{N,s}(a_N)-NI_{s}(\sigma)}{N^s} \leq (2^s-1)\frac{2\zeta(s)}{(2\pi)^s}.\label{eq:limsupleq}
\end{equation}
Combining \eqref{eq:limsupgeq} and \eqref{eq:limsupleq}, we have proved \eqref{eq:limsupunan}. 

We now prove \eqref{eq:liminfunan}. In view of \eqref{halfineqliminf}, we only need to justify
\[
\liminf_{N\rightarrow\infty}\frac{U_{N,s}(a_{N})-N I_{s}(\sigma)}{N^{s}}\geq \overline{g}(s)(2^s-1)\frac{2\zeta(s)}{(2\pi)^{s}}.
\] 
Suppose that $\mathcal{N}$ is a sequence of natural numbers such that 
\[
\left(\frac{U_{N,s}(a_N)-N I_{s}(\sigma)}{N^s}\right)_{N\in\mathcal{N}}\quad\text{converges.}
\]
We will show that its limit is $\geq \overline{g}(s)(2^s-1)\frac{2\zeta(s)}{(2\pi)^{s}}$. First, assume that there exists $p\in\mathbb{N}$ such that
infinitely many elements of $\mathcal{N}$ have binary expansion of length $p$. As in our argument to prove \eqref{eq:limsupunan}, we may choose a subsequence $\widehat{\mathcal{N}}\subset\mathcal{N}$ such that there is a vector $\vec{\theta}=(\theta_1,\ldots,\theta_p)\in\Theta_p$ for which
\[
\lim_{N\in\widehat{\mathcal{N}}}\frac{2^{n_k}}{N}=\theta_k
\]
for all $1\leq k\leq p$. We then have
\begin{align*}
\lim_{N\in\mathcal{N}}\frac{U_{N,s}(a_N)-NI_{s}(\sigma)}{N^s}
&= \lim_{N\in\widehat{\mathcal{N}}}\frac{U_{N,s}(a_N)-NI_{s}(\sigma)}{N^s}\\
&= G(\vec{\theta}; s)(2^s-1)\frac{2\zeta(s)}{(2\pi)^s}\\
&\geq \overline{g}(s)(2^s-1)\frac{2\zeta(s)}{(2\pi)^s}.
\end{align*}

Now suppose that $\tau_{b}(N) \to \infty$ as $N\to\infty$ along the sequence $\mathcal{N}$. 
Let $0<\epsilon<1$ and let $M$ be an integer such that
\[
\sum_{k=M}^{\infty} 2^{-sk} < \epsilon. 
\]
Now, if $N=2^{n_1}+2^{n_2}+\cdots+2^{n_{\tau_{b}(N)}}\in\mathcal{N}$, we write
\[
\frac{U_{N,s}(a_N)-NI_{s}(\sigma)}{N^s} = S_{N,1} + S_{N,2},
\]
where $S_{N,1}$ and $S_{N,2}$ are defined again as in \eqref{def:SN1sl1}--\eqref{def:SN2sl1}. 

Put $\widetilde{N} = 2^{n_1}+\cdots+2^{n_{M}}$. As in our argument to prove \eqref{eq:limsupunan}, we may choose a subsequence $\widehat{\mathcal{N}}\subset\mathcal{N}$ so that there exists a vector $\vec{\theta}=(\theta_1,\ldots,\theta_M)\in\Theta_M$ such that 
\[
\lim_{N\in\widehat{\mathcal{N}}} \frac{2^{n_k}}{\widetilde{N}} = \theta_k
\]
for all $1\leq k\leq M$. 
Since
\begin{align*}
S_{N,1} = \sum_{k=1}^{M}\mathcal{W}_{s}(2^{n_k})\left(\frac{2^{n_k}}{\widetilde{N}}\right)^{s}\left(\frac{\widetilde{N}}{N}\right)^{s} \geq \sum_{k=1}^{M}\mathcal{W}_{s}(2^{n_k})\left(\frac{2^{n_k}}{\widetilde{N}}\right)^{s},
\end{align*}
we have
\begin{align*}
\liminf_{N\in\widehat{\mathcal{N}}} S_{N,1} \geq G(\vec{\theta};s)(2^s-1)\frac{2\zeta(s)}{(2\pi)^s}
\geq \overline{g}(s)(2^s-1)\frac{2\zeta(s)}{(2\pi)^s}.
\end{align*}
Combining this with \eqref{est:SN2}, we obtain
\begin{align*}
\lim_{N\in\widehat{\mathcal{N}}}\frac{U_{N,s}(a_N)-NI_{s}(\sigma)}{N^s} &= \lim_{N\in\widehat{\mathcal{N}}}\left(S_{N,1}+S_{N,2}\right)\geq \liminf_{N\in\widehat{\mathcal{N}}}S_{N,1} + \liminf_{N\in\widehat{\mathcal{N}}} S_{N,2}\\
&\geq \overline{g}(s)(2^s-1)\frac{2\zeta(s)}{(2\pi)^s} - C\epsilon. 
\end{align*}
Since $\epsilon$ was arbitrary, this implies
\[
\lim_{N\in\mathcal{N}} \frac{U_{N,s}(a_N)-NI_{s}(\sigma)}{N^s}\geq \overline{g}(s)(2^s-1)\frac{2\zeta(s)}{(2\pi)^s}.
\]
This concludes the proof of \eqref{eq:liminfunan}.\qed

\begin{figure*}
\centering
\begin{subfigure}[b]{0.475\textwidth}
\centering
\includegraphics[width=\textwidth]{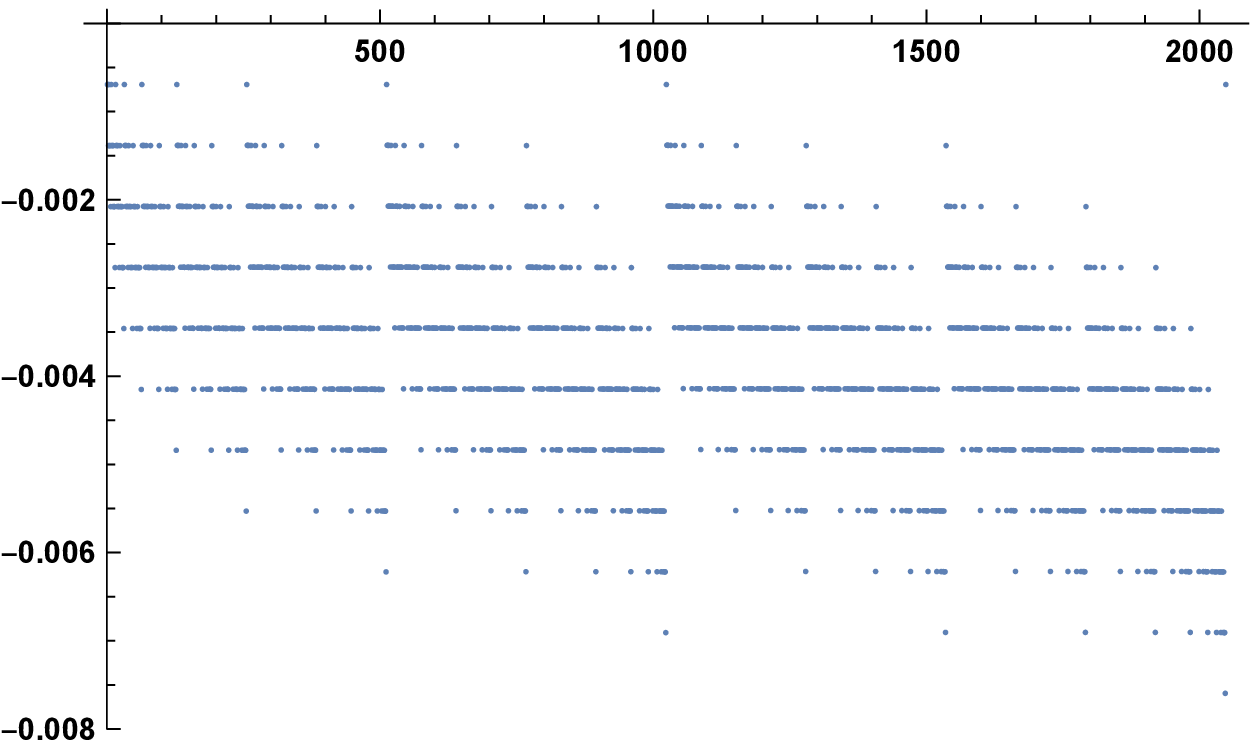}
\caption[]{{$s=0.001$}}    
\end{subfigure}
\hfill
\begin{subfigure}[b]{0.475\textwidth}
\centering 
\includegraphics[width=\textwidth]{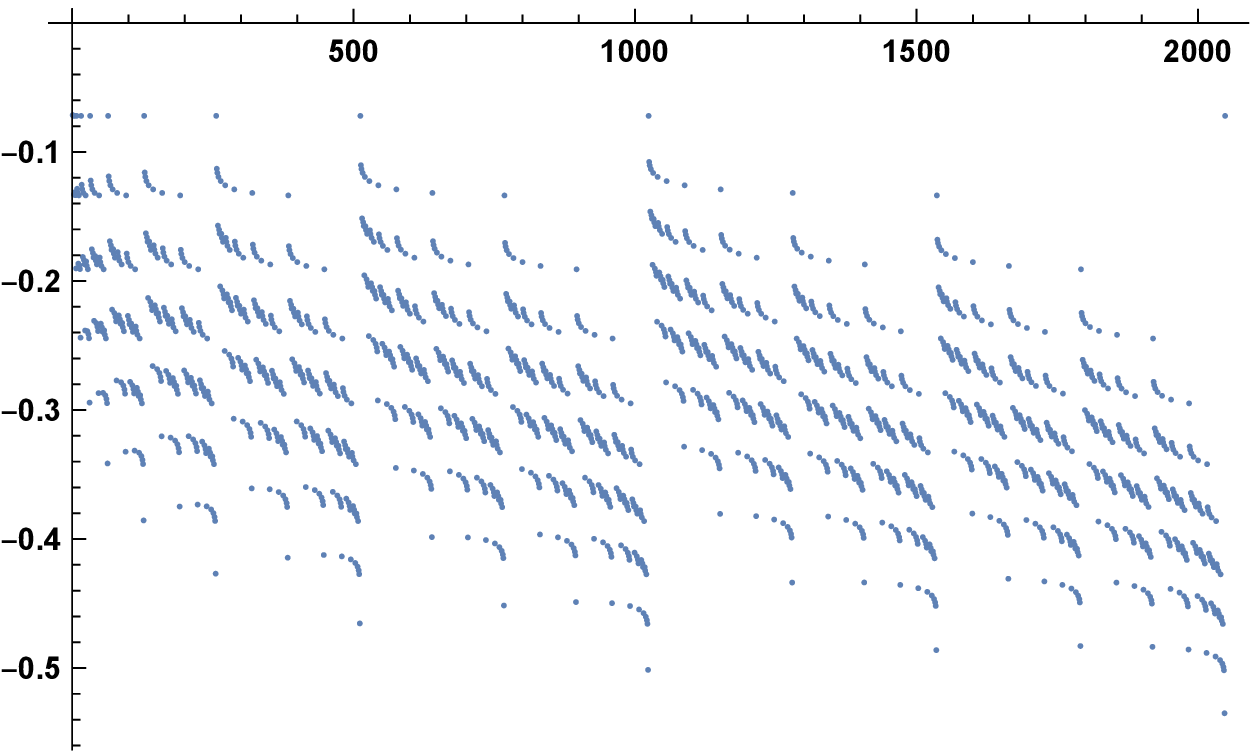}
\caption[]{{$s=0.1$}}
\end{subfigure}
\vskip\baselineskip
\begin{subfigure}[b]{0.475\textwidth}
\centering
\includegraphics[width=\textwidth]{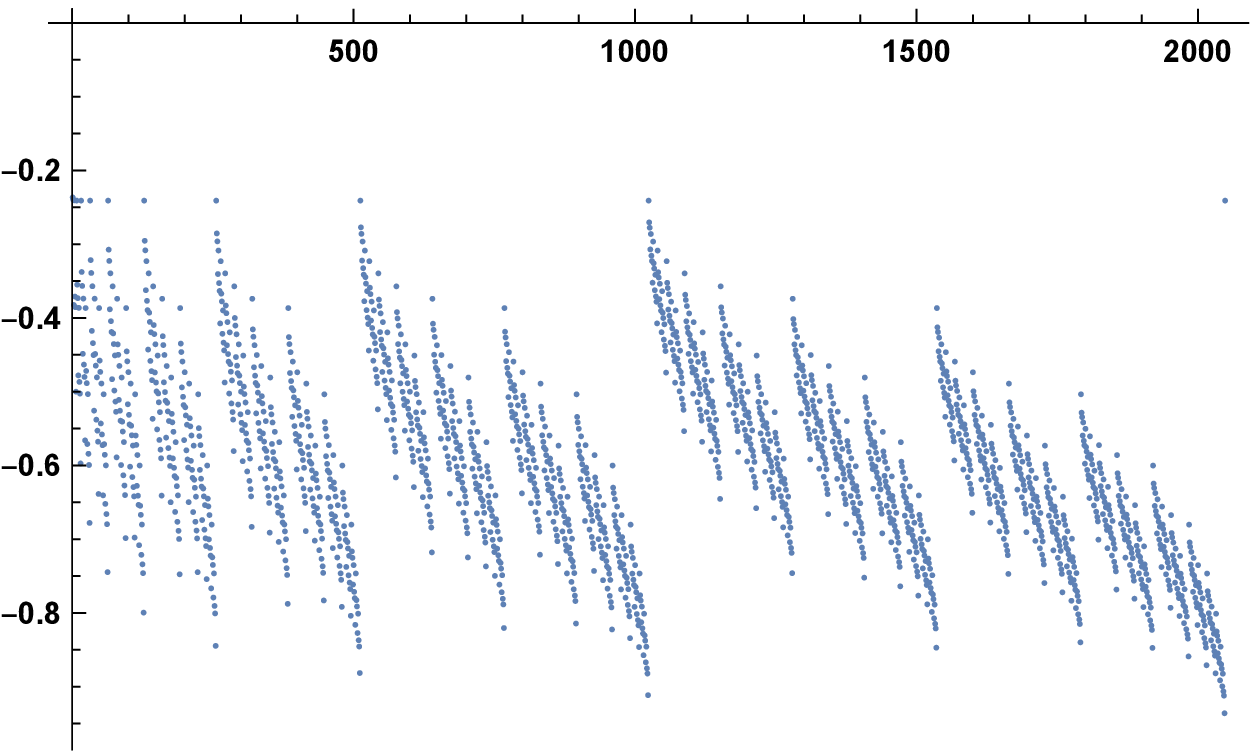}
\caption[]{{$s=0.3$}}    
\end{subfigure}
\hfill
\begin{subfigure}[b]{0.475\textwidth}
\centering 
\includegraphics[width=\textwidth]{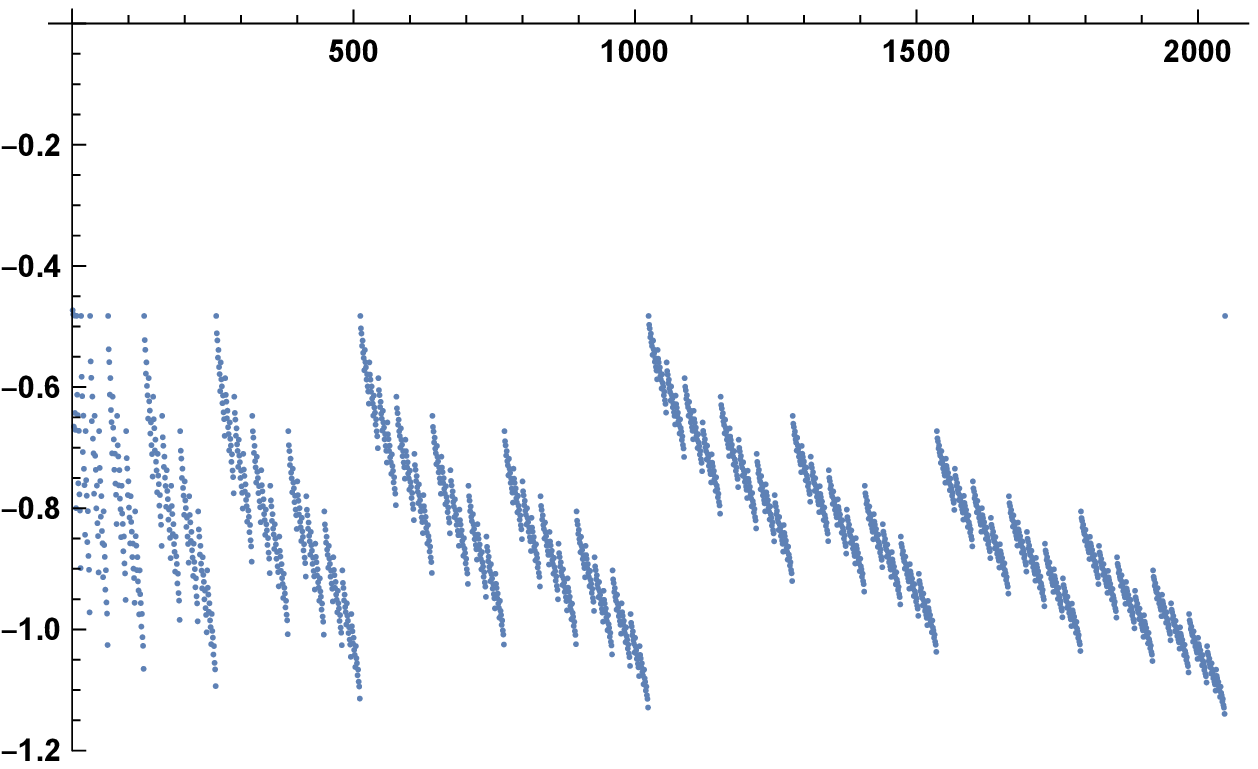}
\caption[]{{$s=0.5$}}
\end{subfigure}
\vskip\baselineskip
\begin{subfigure}[b]{0.475\textwidth}
\centering 
\includegraphics[width=\textwidth]{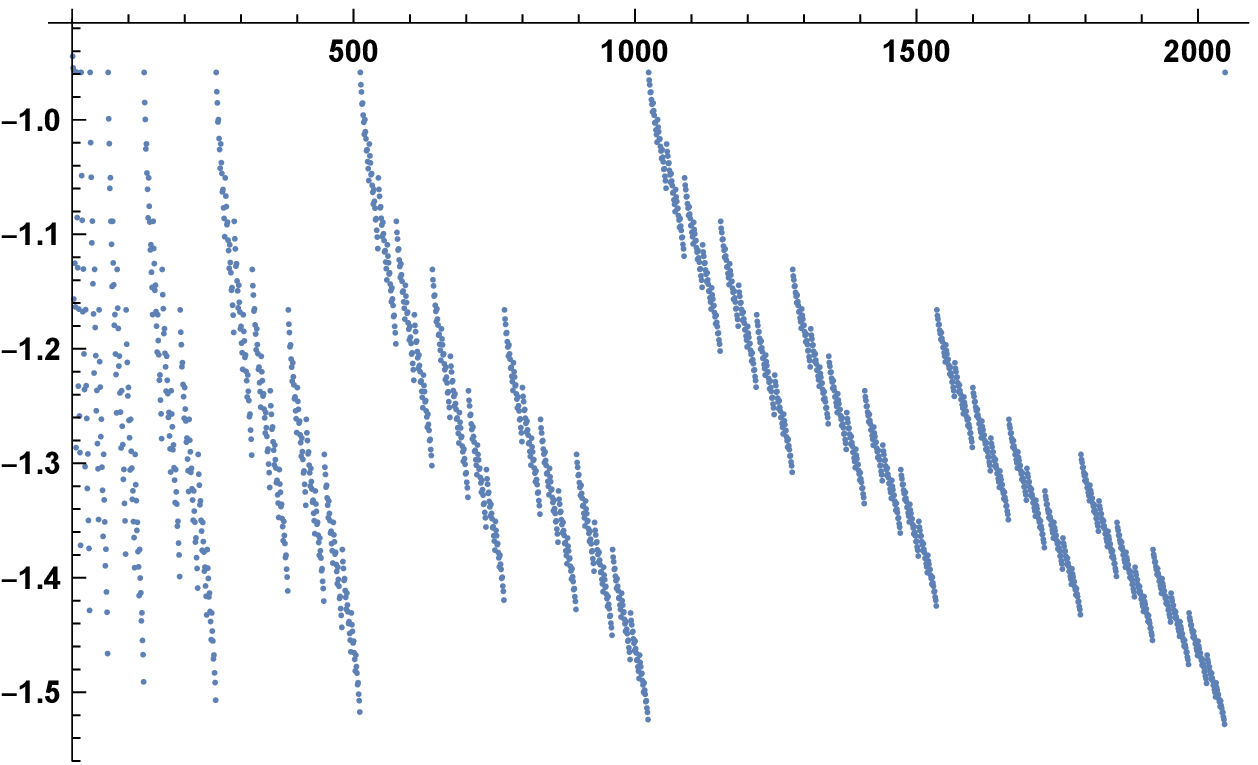}
\caption[]{{$s=0.7$}}    
\end{subfigure}
\quad
\begin{subfigure}[b]{0.475\textwidth}
\centering 
\includegraphics[width=\textwidth]{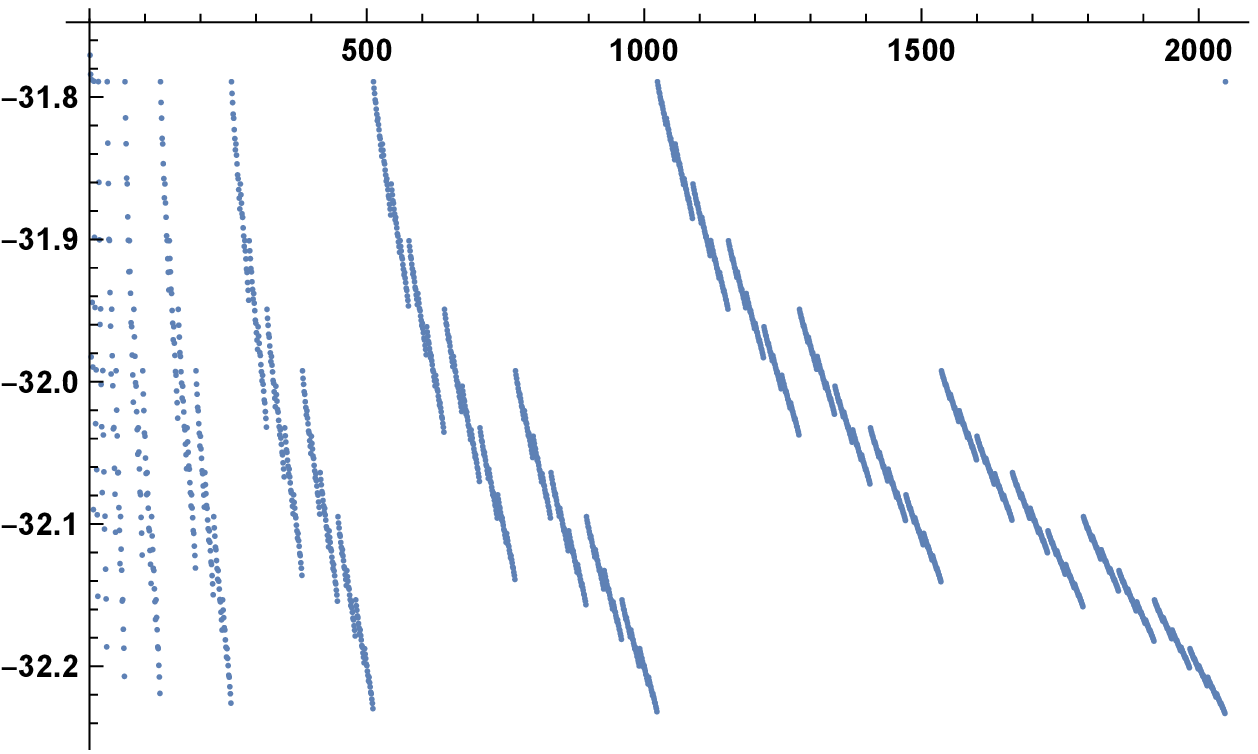}
\caption[]{{$s=0.99$}}    
\end{subfigure}
\caption[]{{Plots of sequences $\left(\frac{U_{N,s}(a_{N})-N I_{s}(\sigma)}{N^{s}}\right)$, $1\leq N\leq 2048$, for different values of $0<s<1$.}}
\label{fig2}
\end{figure*}

\begin{proposition} \label{prop:limPowerOf2Minus1}
Let $N(p)=2^{p}-1$ for $p\in\mathbb{N}$. Then
\begin{equation}\label{eq:specialsublim}
\lim_{p\to\infty} \frac{U_{N(p),s}(a_{N(p)}) - N(p)I_{s}(\sigma)}{(N(p))^{s}} = \frac{2\zeta(s)}{(2\pi)^{s}},\qquad 0<s<1.
\end{equation}
\end{proposition}
\begin{proof}
Since $N(p)=\sum_{k=0}^{p-1}2^{k}$, by \eqref{eq:descUnan} and \eqref{calUsNLsN} we have
\begin{equation}\label{eq:UN(p)}
U_{N(p),s}(a_{N(p)})=\sum_{k=0}^{p-1} \left(\frac{\mathcal{L}_{s}(2^{k+1})}{2^{k+1}}-\frac{\mathcal{L}_{s}(2^k)}{2^k}\right)= \frac{\mathcal{L}_{s}(2^p)}{2^p} - \frac{\mathcal{L}_{s}(1)}{1}=\frac{\mathcal{L}_{s}(2^p)}
{2^{p}}
\end{equation}
recall that $\mathcal{L}_{s}(1)=0$. Hence,
\begin{align*}
\frac{U_{N(p),s}(a_{N(p)}) - N(p)I_{s}(\sigma)}{(N(p))^{s}} &= \frac{\frac{\mathcal{L}_{s}(2^p)}{2^p} - (2^p - 1)I_{s}(\sigma)}{(2^p-1)^{s}}\\
&= \frac{\mathcal{L}_{s}(2^p)}{2^p(2^p-1)^{s}} - \frac{2^p}{(2^p - 1)^s}I_{s}(\sigma) + (2^p-1)^{-s}I_{s}(\sigma) \\
&= \frac{\mathcal{L}_{s}(2^p) - (2^p)^2I_{s}(\sigma)}{2^p(2^p-1)^s} + (2^p-1)^{-s}I_{s}(\sigma).
\end{align*}
Since $\lim_{N \to \infty}\mathcal{R}_{s}(N)= 2\zeta(s)/(2\pi)^{s}$, see \eqref{def:RsN} and \eqref{eq:asympRsN}, we obtain
\[
\lim_{p\to\infty} \frac{U_{N(p),s}(a_{N(p)}) - N(p)I_{s}(\sigma)}{(N(p))^s} = \frac{2\zeta(s)}{(2\pi)^s}.
\]
\end{proof}

\begin{remark}
From \eqref{eq:specialsublim} we deduce that 
\[
\liminf_{N\rightarrow\infty}\frac{U_{N,s}(a_{N})-N I_{s}(\sigma)}{N^{s}}\leq \frac{2\zeta(s)}{(2\pi)^s}.
\]
Comparing this with \eqref{eq:liminfunan}, we obtain
\[
\overline{g}(s)\geq \frac{1}{2^{s}-1},\qquad 0<s<1.
\]
\end{remark}

\section{The case $s=1$}\label{sec:sequalone}

We start with the following auxiliary result.

\begin{proposition}
We have
\[
\lim_{N\rightarrow\infty}\frac{\mathcal{U}_{1}(N)}{N\log N}=\frac{1}{\pi}.
\]
\end{proposition} 
\begin{proof}
The following formula is deduced from \cite[Thm. 3.2]{MMRS}:
\[
\lim_{N\to\infty} \frac{\mathcal{L}_{1}(N)}{N^2\log N}=\frac{1}{\pi}.
\]
We have
\[
\frac{\mathcal{U}_{1}(N)}{N\log N} = \frac{\mathcal{L}_{1}(2N)}{2N^2 \log N} - \frac{\mathcal{L}_{1}(N)}{N^2\log N}
\]
and since $2N^2\log N = \frac{1}{2}(2N)^2\log(2N) - 2\log(2)N^2$, we get
\[
\lim_{N \to \infty} \frac{\mathcal{L}_{1}(2N)}{2N^2 \log N} = \lim_{N \to \infty} \frac{\mathcal{L}_{1}(2N)}{\frac{1}{2}(2N)^2\log(2N)} = \frac{2}{\pi}.
\]
Hence,
\begin{align*}
\lim_{N \to \infty} \frac{\mathcal{U}_{1}(N)}{N\log N} = \lim_{N \to \infty} \left(\frac{\mathcal{L}_{1}(2N)}{2N^2\log N} - \frac{\mathcal{L}_{1}(N)}{N^2\log N}\right) = \frac{2}{\pi} - \frac{1}{\pi} = \frac{1}{\pi}.
\end{align*}
\end{proof}

In this section we will use the notation
\[
\mathcal{W}_{1}(N):= \frac{\mathcal{U}_{1}(N)}{N\log N},\qquad N\geq 2.
\]
Therefore we have
\begin{equation}\label{asympW1N}
\lim_{N\to\infty} \mathcal{W}_{1}(N) = \frac{1}{\pi}.
\end{equation}

\noindent\textit{Proof of Theorem}~\ref{theo:foase1}. Let $N=2^{n_1}+2^{n_2}+\cdots+2^{n_{p}}$ where $n_1>n_2>\cdots>n_p\geq 0$. Then
\begin{align*}
\frac{U_{N,1}(a_N)}{N\log N}
= \sum_{k=1}^{p}\frac{\mathcal{U}_{1}(2^{n_k})}{N\log N}
= \sum_{k=1}^{p}\mathcal{W}_{1}(2^{n_k})\frac{2^{n_k}\log 2^{n_k}}{N\log N}.
\end{align*}
Observe that
\[
\sum_{k=1}^{p}2^{n_{k}} \log 2^{n_{k}}\leq \sum_{k=1}^{p}2^{n_{k}}\log N=N\log N,
\]
so the sequence $(U_{N,1}(a_{N})/N\log N)$ is bounded.

Fix $p\in\mathbb{N}$ and $\vec{\theta}=(\theta_1,\theta_2,\ldots,\theta_p)\in\Theta_p$. 
Let $\mathcal{N}$ be a sequence of natural numbers $N = 2^{n_1} + 2^{n_2} + \cdots + 2^{n_p}$, $n_1>n_2>\cdots >n_p\geq 0$, such that
\[
\lim_{N\in\mathcal{N}}\frac{2^{n_k}}{N} = \theta_k,\qquad 1\leq k\leq p.
\]
Fix $1\leq k\leq p$. First, assume that $\theta_{k}>0$.
Since
\[
\frac{\log 2^{n_k}}{\log N} = \frac{\log 2^{n_k}-\log N+\log N}{\log N}
= \frac{\log(2^{n_k}/N)}{\log N} + 1
\]
and $2^{n_k}/N\to\theta_k$, we have in this case
\[
\lim_{N\in\mathcal{N}}\frac{\log 2^{n_k}}{\log N} = 1.
\]
Therefore by \eqref{asympW1N} we get
\[
\lim_{N\in\mathcal{N}}\mathcal{W}_{1}(2^{n_k})\frac{2^{n_k}\log 2^{n_k}}{N\log N} = \frac{\theta_k}{\pi}.
\]
Now assume that $\theta_k = 0$. Since $2^{n_k}\leq N$, we have that 
$\log 2^{n_k}/\log N$ is bounded, thus
\[
\lim_{N\in\mathcal{N}}\mathcal{W}_{1}(2^{n_k})\frac{2^{n_k}\log 2^{n_k}}{N\log N} = 0.
\]
But since $\theta_k=0$, this limit is also equal to $\theta_k/\pi$ in this case. 

It then follows that
\begin{equation}\label{eq:limUn1:1}
\lim_{N\in\mathcal{N}} \frac{U_{N,1}(a_N)}{N\log N}
= \lim_{N\in\mathcal{N}}\sum_{k=1}^{p} \mathcal{W}_{1}(2^{n_k})\frac{2^{n_k}\log 2^{n_k}}{N\log N}
= \sum_{k=1}^{p}\frac{\theta_k}{\pi} = \frac{1}{\pi}.
\end{equation}

To prove \eqref{eq:firstorders1}, we will show that any convergent subsequence
\begin{equation}\label{seqUn1}
\left(\frac{U_{N,1}(a_{N})}{N\log N}\right)_{N\in\mathcal{N}}
\end{equation}
of the sequence under analysis has limit $1/\pi$. Let $\mathcal{N}$ be such a subsequence. 

If there exists $p\in\mathbb{N}$ such that $\tau_{b}(N)=p$ for infinitely many $N\in\mathcal{N}$, then we may choose a subsequence $\widehat{\mathcal{N}}\subset\mathcal{N}$ associated with a vector $(\theta_{1},\ldots,\theta_{p})\in\Theta_{p}$, and our previous argument leading to \eqref{eq:limUn1:1} would show that the sequence \eqref{seqUn1} converges to $1/\pi$.

Now assume that $\tau_{b}(N)\rightarrow\infty$ along the sequence $\mathcal{N}$. Let $\epsilon>0$. Let $M$ be an integer such that 
\begin{equation}\label{newsmalltail}
\sum_{k=M}^{\infty} 2^{-k} < \epsilon. 
\end{equation}
For $N=2^{n_{1}}+2^{n_{2}}+\cdots+2^{n_{\tau_{b}(N)}}\in\mathcal{N}$ we write
\[
\frac{U_{N,1}(a_N)}{N\log N} = S_{N,1} + S_{N,2},
\]
where
\begin{align*}
S_{N,1} & :=\sum_{k=1}^{M}\mathcal{W}_{1}(2^{n_k})\frac{2^{n_k}\log 2^{n_k}}{N\log N},\\
S_{N,2} & :=\sum_{k=M+1}^{\tau_{b}(N)} \mathcal{W}_{1}(2^{n_k})\frac{2^{n_k}\log 2^{n_k}}{N\log N}.
\end{align*}
The second sum is understood to be zero if $M\geq \tau_{b}(N)$. 
Since $(\mathcal{W}_{1}(N))$ converges, there is a constant $C$ such that $|\mathcal{W}_{1}(N)|\leq C$ for all $N$. So we have
\begin{equation}\label{eq:boundSn2s1}
\left|S_{N,2}\right| \leq C\sum_{k=M+1}^{\tau_{b}(N)} 
\frac{2^{n_k}\log 2^{n_{k}}}{N \log N}\leq C\sum_{k=M+1}^{\tau_{b}(N)} 
2^{n_k-n_{1}}<C\sum_{k=M}^{\infty} 2^{-k} < C\epsilon. 
\end{equation}
We now turn our attention to $S_{N,1}$. 

Let $\widetilde{N}:= 2^{n_1} + 2^{n_2} + \cdots + 2^{n_M}$. 
If necessary, we may choose a subsequence $\widehat{\mathcal{N}}\subset\mathcal{N}$ such that
\[
\lim_{N\in\widehat{\mathcal{N}}}\frac{2^{n_k}}{\widetilde{N}} = \theta_k
\]
for all $1\leq k\leq M$ and for some $\vec{\theta}=(\theta_1,\theta_2,\ldots,\theta_M)\in\Theta_M$. 

Note that
\[
\frac{\widetilde{N}\log\widetilde{N}}{N\log N}
= \frac{\widetilde{N}}{N}\frac{\log(\widetilde{N}/N)}{\log N}+\frac{\widetilde{N}}{N}.
\]
Since
\[
\frac{\widetilde{N}}{N} = 1-\frac{2^{n_{M+1}}+\cdots+2^{n_{\tau_{b}(N)}}}{N} > 1 - \epsilon,
\]
we have
\begin{equation}\label{closeNN1}
\left|\frac{\widetilde{N}}{N}-1\right|<\epsilon.
\end{equation}
Hence,
\begin{align*}
\left|\frac{\widetilde{N}\log\widetilde{N}}{N\log N}-1\right|
&= \left|\frac{\widetilde{N}}{N}\frac{\log\left(\widetilde{N}/N\right)}{\log N}+\frac{\widetilde{N}}{N}-1\right| \\
&\leq \left|\frac{\widetilde{N}}{N}\frac{\log\left(\widetilde{N}/N\right)}{\log N}\right|
+ \left|\frac{\widetilde{N}}{N}-1\right|\\
&< |\log(1-\epsilon)|+\epsilon.
\end{align*}

Suppose $N\in\widehat{\mathcal{N}}$ is large enough so that
\begin{align*}
\left|\sum_{k=1}^{M}\mathcal{W}_{1}(2^{n_k})\frac{2^{n_k}\log 2^{n_k}}{\widetilde{N}\log\widetilde{N}} - \frac{1}{\pi}\right| < \epsilon. 
\end{align*}
For such $N$ we then have
\begin{align*}
\left|S_{N,1}-\frac{1}{\pi}\right|& =\left|\sum_{k=1}^{M}\mathcal{W}_{1}(2^{n_k})\frac{2^{n_k}\log 2^{n_k}}{\widetilde{N}\log\widetilde{N}}\frac{\widetilde{N}\log\widetilde{N}}{N\log N} - \frac{1}{\pi}\right|\\
& =\left|\sum_{k=1}^{M}\mathcal{W}_{1}(2^{n_k})\frac{2^{n_k}\log 2^{n_k}}{\widetilde{N}\log\widetilde{N}}\frac{\widetilde{N}\log\widetilde{N}}{N\log N} - \frac{1}{\pi}\frac{\widetilde{N}\log\widetilde{N}}{N\log N} + \frac{1}{\pi}\frac{\widetilde{N}\log\widetilde{N}}{N\log N}- \frac{1}{\pi}\right|\\
&\leq \left|\sum_{k=1}^{M}\mathcal{W}_{1}(2^{n_k})\frac{2^{n_k}\log 2^{n_k}}{\widetilde{N}\log\widetilde{N}}-\frac{1}{\pi}\right|
+ \frac{1}{\pi}\left|\frac{\widetilde{N}\log\widetilde{N}}{N\log N}-1\right|\\
&< \epsilon + \frac{1}{\pi}\left(|\log(1-\epsilon)|+\epsilon\right).
\end{align*}

Using \eqref{eq:boundSn2s1}, it then follows that
\begin{align*}
\left|\frac{U_{N,1}(a_{N})}{N\log N}-\frac{1}{\pi}\right|
& =\left|S_{N,1} + S_{N,2} - \frac{1}{\pi}\right|\leq \left|S_{N,1}-\frac{1}{\pi}\right| + \left|S_{N,2}\right|\\ 
& <\epsilon + \frac{1}{\pi}\left(|\log(1-\epsilon)|+\epsilon\right)+C \epsilon.
\end{align*}

Since the sequence \eqref{seqUn1} converges, and $\epsilon$ is arbitrary, we have shown that
\[
\lim_{N\in\mathcal{N}}\frac{U_{N,1}(a_N)}{N\log N} = \frac{1}{\pi}. 
\]
This concludes the proof.\qed

\smallskip

We establish now bounds for the function $\Lambda(\vec{\theta})$ defined in \eqref{def:Lambdafunc}.

\begin{lemma}\label{lembound}
There exists a constant $C<0$ such that for all $\vec{\theta}\in\Theta$, 
\[
C < \Lambda(\vec{\theta}) \leq 0. 
\]
\end{lemma}

\begin{proof}
Fix a vector $\vec{\theta}=(\theta_{1},\ldots,\theta_{p})\in\Theta$. Since $0\leq \theta_k\leq 1$ for all $k$, we have
\[
\Lambda(\vec{\theta}) \leq 0.
\]
Note that the function $x \mapsto x\log x$ has an absolute minimum value of $-e^{-1}$, and it is decreasing on the interval $0<x<e^{-1}$. Let $M$ be large enough so that $2^{-M} < e^{-1}$. 
If $p\leq M$, then we have the trivial bound
\[
\Lambda(\vec{\theta}) = \sum_{k=1}^{p} \theta_k\log\theta_k \geq -\sum_{k=1}^{p}e^{-1}\geq-Me^{-1}. 
\]
Now suppose that $p \geq M+1$. We will use the inequality $\theta_k \leq 2^{-(k-1)}$ and our choice of $M$ to obtain the following:
\begin{align*}
\Lambda(\vec{\theta}) &= \sum_{k=1}^{M}\theta_k\log\theta_k + \sum_{k=M+1}^{p} \theta_k\log\theta_k
\geq -Me^{-1} + \sum_{k=M+1}^{p}2^{-(k-1)}\log 2^{-(k-1)}\\
&= -Me^{-1} - \log(2)\sum_{k=M+1}^{p} (k-1)2^{-(k-1)}\\
&> -Me^{-1} - \log(2) \sum_{n=1}^{\infty} n 2^{-n} 
= -Me^{-1} - 2\log 2.
\end{align*}
Thus, we may take $C = -Me^{-1} - 2\log 2$. 
\end{proof}

\begin{proposition}
The following limit holds:
\begin{equation}\label{asympcalU1Nnorm}
\lim_{N\to\infty}\frac{\mathcal{U}_{1}(N)-\frac{1}{\pi}N\log N}{N}= \frac{1}{\pi}\left(\gamma+\log\left(\frac{8}{\pi}\right)\right),
\end{equation}
where $\gamma$ is the Euler-Mascheroni constant.
\end{proposition}
\begin{proof}
The following formula is deduced from \cite[Thm. 1.3]{BrauHardSaff}
\begin{equation}\label{limauxBHS}
\lim_{N\to\infty}\frac{\mathcal{L}_{1}(N)-\frac{1}{\pi}N^2\log N}{N^2} = \frac{1}{\pi}(\gamma+\log\left(2/\pi\right)).
\end{equation}
If $N\geq 1$, then
\begin{align*}
\frac{\mathcal{U}_{1}(N)-\frac{1}{\pi}N\log N}{N}
&= \frac{\mathcal{L}_{1}(2N)}{\frac{1}{2}(2N)^2}
- \frac{\mathcal{L}_{1}(N)}{N^2} - \frac{1}{\pi}\log N\\
&= 2\frac{\mathcal{L}_{1}(2N)-\frac{1}{\pi}(2N)^2\log 2N+\frac{1}{\pi}(2N)^2\log 2N}{(2N)^2}\\
&- \frac{\mathcal{L}_{1}(N)-\frac{1}{\pi}N^2\log N+\frac{1}{\pi}N^2\log N}{N^2} - \frac{1}{\pi}\log N\\
&= 2\frac{\mathcal{L}_{1}(2N)-\frac{1}{\pi}(2N)^2\log 2N}{(2N)^2}
- \frac{\mathcal{L}_{1}(N)-\frac{1}{\pi}N^2\log N}{N^2}+ \frac{1}{\pi}\log 4.
\end{align*}
The claim now follows from \eqref{limauxBHS}.
\end{proof}

\begin{figure}
\centering
\includegraphics[width=\linewidth]{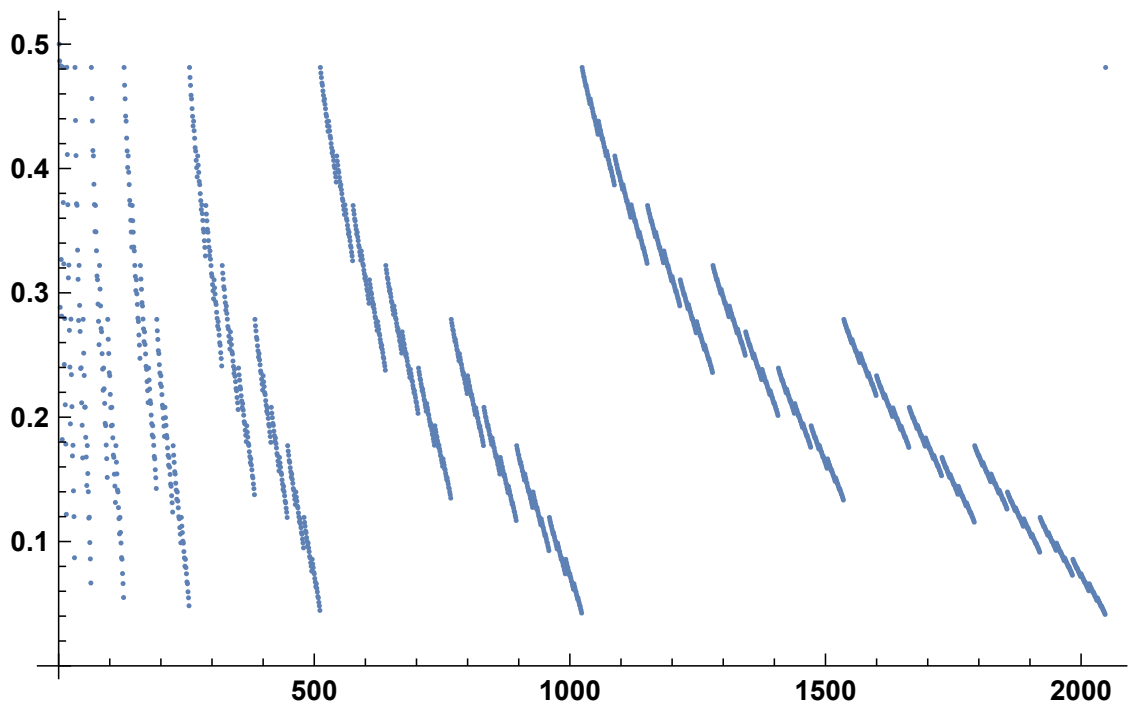}
\caption[]{Plot of the sequence $\left(\frac{U_{N,1}(a_{N})-\frac{1}{\pi}N\log N}{N}\right)$ for $1 \leq N \leq 2048$.}
\label{fig3}
\end{figure}

\noindent\textit{Proof of Theorem}~\ref{theo:soase1}. We will use the notation
\[
\mathcal{T}(N):=\frac{\mathcal{U}_{1}(N)-\frac{1}{\pi}N\log N}{N},\qquad N\geq 1.
\]
So according to \eqref{asympcalU1Nnorm} we have
\begin{equation}\label{asympcalTN}
\lim_{N\rightarrow\infty}\mathcal{T}(N)=\frac{1}{\pi}\left(\gamma+\log\left(\frac{8}{\pi}\right)\right).
\end{equation}

For any index $N=2^{n_{1}}+\cdots+2^{n_{p}}$, we have
\begin{align*}
\frac{U_{N,1}(a_N)-\frac{1}{\pi}N\log N}{N}
&= \sum_{k=1}^{p}\frac{\mathcal{U}_{1}(2^{n_k})-\frac{1}{\pi}2^{n_k}\log 2^{n_k} - \frac{1}{\pi}2^{n_k}\log\left(2^{-n_k}N\right)}{N}\\
&= \sum_{k=1}^{p}\frac{2^{n_k}}{N}\left(\mathcal{T}(2^{n_k}) + \frac{1}{\pi}\log\left(\frac{2^{n_k}}{N}\right)\right).
\end{align*}
The boundedness of \eqref{normsecordseqse1} then follows immediately from the boundedness of $(\mathcal{T}(N))$ and Lemma~\ref{lembound}.
 
Fix $\vec{\theta}=(\theta_1,\ldots,\theta_p)\in\Theta_p$, for some $p\in\mathbb{N}$. Let $\mathcal{N}$ be a sequence of natural numbers $N=2^{n_1}+2^{n_2}+\cdots+2^{n_p}$, $n_1>n_2>\cdots>n_p\geq 0$, such that 
\begin{equation}\label{condseqtheta}
\lim_{N\in\mathcal{N}} \frac{2^{n_{k}}}{N} = \theta_k,\qquad 1\leq k\leq p.
\end{equation}
Fix $1\leq k\leq p$. First, suppose $\theta_k > 0$. It then follows from \eqref{asympcalTN} and \eqref{condseqtheta} that
\begin{align*}
\lim_{N\in\mathcal{N}} \frac{2^{n_k}}{N}\left(\mathcal{T}(2^{n_k}) + \frac{1}{\pi}\log\left(\frac{2^{n_k}}{N}\right)\right)
= \frac{\theta_k}{\pi}\left(\gamma+\log\left(\frac{8}{\pi}\right)+\log\theta_k\right).
\end{align*}
If $\theta_k=0$, then clearly 
\begin{align*}
\lim_{N\in\mathcal{N}} \frac{2^{n_k}}{N}\left(\mathcal{T}(2^{n_k}) + \frac{1}{\pi}\log\left(\frac{2^{n_k}}{N}\right)\right)=0.
\end{align*}
Hence,
\begin{align}
\lim_{N\in\mathcal{N}}\frac{U_{N,1}(a_N)-\frac{1}{\pi}N\log N}{N}&=\lim_{N\in\mathcal{N}}\sum_{k=1}^{p}\frac{2^{n_k}}{N}\left(\mathcal{T}(2^{n_k}) + \frac{1}{\pi}\log\left(\frac{2^{n_k}}{N}\right)\right)\notag\\
&=\sum_{k=1}^{p}\frac{\theta_k}{\pi}\left(\gamma+\log\left(\frac{8}{\pi}\right)+\log\theta_k\right)\notag\\
&= \frac{1}{\pi}\left(\gamma+\log\left(\frac{8}{\pi}\right)+\Lambda(\vec{\theta})\right).\label{impcons}
\end{align}
Therefore the value \eqref{impcons} is a limit point of the sequence \eqref{normsecordseqse1}. Since $\vec{\theta}$ is arbitrary, it also follows that 
\begin{equation}\label{ineqsecordliminfs1}
\liminf_{N\rightarrow\infty}\frac{U_{N,1}(a_N)-\frac{1}{\pi}N\log N}{N}\leq \frac{1}{\pi}\left(\gamma+\log\left(\frac{8}{\pi}\right)+\lambda\right).
\end{equation}
We also deduce from \eqref{impcons} that for sequences $\mathcal{N}$ associated with a vector $\vec{\theta}\in\Theta_{p}$, we have
\[
\lim_{N\in\mathcal{N}}\frac{U_{N,1}(a_N)-\frac{1}{\pi}N\log N}{N}\leq \frac{1}{\pi}\left(\gamma+\log\left(\frac{8}{\pi}\right)\right).
\]

Now we turn to the proof of \eqref{eq:secordlimsups1}. If we take the sequence $N=2^{n}$, then
\[
\lim_{n\rightarrow\infty}\frac{U_{2^{n},1}(a_{2^{n}})-\frac{1}{\pi} 2^{n}\log 2^{n}}{2^{n}}=\lim_{n\rightarrow\infty}\mathcal{T}(2^{n})=\frac{1}{\pi}\left(\gamma + \log\left(\frac{8}{\pi}\right)\right),
\]
thus
\[
\limsup_{N\rightarrow\infty}\frac{U_{N,1}(a_N)-\frac{1}{\pi}N\log N}{N}\geq \frac{1}{\pi}\left(\gamma+\log\left(\frac{8}{\pi}\right)\right).
\]
In order to prove the reverse inequality, we consider an arbitrary sequence $\mathcal{N}$ of natural numbers such that
\begin{equation}\label{subseqconv}
\left(\frac{U_{N,1}(a_N)-\frac{1}{\pi}N\log N}{N}\right)_{N\in\mathcal{N}}\qquad\text{converges,}
\end{equation}
and we will show that its limit is not greater than $(1/\pi)(\gamma+\log(8/\pi))$. First, suppose there exists $p\in\mathbb{N}$ such that $\tau_{b}(N)=p$ for infinitely many $N\in\mathcal{N}$. If necessary, we may choose a subsequence $\widehat{\mathcal{N}}\subset\mathcal{N}$ such that 
for $N=2^{n_1}+\cdots+2^{n_p}\in\widehat{\mathcal{N}}, n_1>\cdots>n_p\geq 0$, we have
\[
\lim_{N\in\widehat{\mathcal{N}}}\frac{2^{n_k}}{N}=\theta_k,\qquad 1\leq k\leq p,
\]
for some $\vec{\theta}=(\theta_1,\ldots,\theta_p)\in\Theta_p$. Then
\[
\lim_{N\in\mathcal{N}}\frac{U_{N,1}(a_N)-\frac{1}{\pi}N\log N}{N}=\lim_{N\in\widehat{\mathcal{N}}} \frac{U_{N,1}(a_N)-\frac{1}{\pi}N\log N}{N}\leq \frac{1}{\pi}\left(\gamma + \log\left(\frac{8}{\pi}\right)\right). 
\]
The above inequality follows from the fact that $\widehat{\mathcal{N}}$ is a sequence corresponding to some $\vec{\theta}\in\Theta_p$. 

Now, suppose that $\tau_{b}(N)\to\infty$ as $N\to\infty$ along the sequence $\mathcal{N}$. Fix $0<\epsilon<1$. Let $M$ be an integer large enough such that 
\begin{equation}\label{twocondM}
2^{-M}<e^{-1}\qquad \mbox{and}\qquad\sum_{k=M}^{\infty} k\,2^{-k} < \epsilon.
\end{equation}
For $N=2^{n_{1}}+\cdots+2^{n_{\tau_{b}(N)}}\in\mathcal{N}$, we define
\begin{align*}
S_{N,1} & := \sum_{k=1}^{M}\frac{2^{n_k}}{N}\left(\mathcal{T}(2^{n_k}) + \frac{1}{\pi}\log\left(\frac{2^{n_k}}{N}\right)\right)\\
S_{N,2} & := \sum_{k=M+1}^{\tau_{b}(N)}\frac{2^{n_k}}{N}\left(\mathcal{T}(2^{n_k}) + \frac{1}{\pi}\log\left(\frac{2^{n_k}}{N}\right)\right)
\end{align*}
where $S_{N,2}$ is understood to be zero if $M\geq\tau_{b}(N)$. 

We analyze $S_{N,2}$ first. Since $(\mathcal{T}(N))$ converges, there is a constant $C$ such that $|\mathcal{T}(N)|\leq C$ for all $N$. For $k\geq M+1$ we have
\[
\frac{2^{n_{k}}}{N}\leq 2^{-(k-1)}\leq 2^{-M}<e^{-1}
\]
and the function $x\mapsto |x\log x|$ is increasing for $0<x<e^{-1}$, so it follows from \eqref{twocondM} that
\begin{align}
|S_{N,2}| &= \left|\sum_{k=M+1}^{\tau_{b}(N)}\frac{2^{n_k}}{N}\left(\mathcal{T}(2^{n_k}) + \frac{1}{\pi}\log\left(\frac{2^{n_k}}{N}\right)\right)\right|\notag\\
&\leq C\sum_{k=M+1}^{\tau_{b}(N)}\frac{2^{n_k}}{N}
+ \frac{1}{\pi}\sum_{k=M+1}^{\tau_{b}(N)}\left|\frac{2^{n_k}}{N}\log\left(\frac{2^{n_k}}{N}\right)\right|\notag\\
& \leq C\sum_{k=M+1}^{\tau_{b}(N)}2^{-(k-1)}
+ \frac{1}{\pi}\sum_{k=M+1}^{\tau_{b}(N)}\left|2^{-(k-1)}\log(2^{-(k-1)})\right|\notag\\
&< (C + \frac{1}{\pi}\log 2)\,\epsilon.\label{aboundforSN2}
\end{align}

We turn our attention now to $S_{N,1}$. We define $\widetilde{N}:=2^{n_{1}}+\cdots+2^{n_{M}}$, truncation of the binary expansion of $N$. We choose a subsequence $\widehat{\mathcal{N}}\subset\mathcal{N}$ such that for $N=2^{n_1}+\cdots+2^{n_{\tau_{b}(N)}}\in\widehat{\mathcal{N}}$ we have
\begin{equation}\label{subgentheta}
\lim_{N\in\widehat{\mathcal{N}}}\frac{2^{n_k}}{\widetilde{N}}=\theta_k,\qquad 1\leq k\leq M,
\end{equation}
for some $\vec{\theta}=(\theta_1,\ldots,\theta_M)\in\Theta_M$. In what follows, let $L:=(1/\pi)(\gamma+\log(8/\pi))$. We write
\begin{equation}\label{SN1part1}
\sum_{k=1}^{M}\frac{2^{n_{k}}}{N}\mathcal{T}(2^{n_{k}})-L=\sum_{k=1}^{M}\frac{2^{n_{k}}}{\widetilde{N}}\Big(\frac{\widetilde{N}}{N}\mathcal{T}(2^{n_{k}})-L\Big).
\end{equation}
Fix $1\leq k\leq M$. If $\theta_{k}=0$, then \eqref{subgentheta} implies 
\begin{equation}\label{limittheta1}
\lim_{N\in\widehat{\mathcal{N}}}\frac{2^{n_{k}}}{\widetilde{N}}\Big(\frac{\widetilde{N}}{N}\mathcal{T}(2^{n_{k}})-L\Big)=0.
\end{equation}
Assume now that $\theta_{k}>0$. First, observe that $\widetilde{N}$ tends to infinity as $N$ tends to infinity along the subsequence $\widehat{\mathcal{N}}$ (this follows from the fact that $\widetilde{N}\geq 2^{n_{1}}\geq 2^{\tau_{b}(N)-1}\rightarrow\infty$). Therefore, from $\theta_{k}>0$ and \eqref{subgentheta} we obtain that $\mathcal{T}(2^{n_{k}})\rightarrow L$ as $N\rightarrow\infty$ along $\widehat{\mathcal{N}}$. Observe also that the estimate \eqref{closeNN1} is valid, see \eqref{twocondM} and \eqref{newsmalltail}. Using the triangle inequality it easily follows that
\begin{equation}\label{limittheta2}
\limsup_{N\in\widehat{\mathcal{N}}}\frac{2^{n_{k}}}{\widetilde{N}}\left|\frac{\widetilde{N}}{N}\mathcal{T}(2^{n_{k}})-L\right|\leq L\,\theta_{k}\,\epsilon.
\end{equation}
We conclude from \eqref{SN1part1}, \eqref{limittheta1}, and \eqref{limittheta2} that 
\begin{equation}\label{estfirstpart}
\limsup_{N\in\widehat{\mathcal{N}}}\left|\sum_{k=1}^{M}\frac{2^{n_{k}}}{N}\mathcal{T}(2^{n_{k}})-L\right|\leq L\,\epsilon.
\end{equation}

Fix $1\leq k\leq M$ again. We can write
\begin{align*}
\frac{2^{n_{k}}}{N}\log\left(\frac{2^{n_{k}}}{N}\right)-\theta_{k}\log \theta_{k} & 
=\frac{2^{n_{k}}}{N}\log\left(\frac{2^{n_{k}}}{N}\right)-\frac{2^{n_{k}}}{\widetilde{N}}\log\left(\frac{2^{n_{k}}}{\widetilde{N}}\right)+\frac{2^{n_{k}}}{\widetilde{N}}\log\left(\frac{2^{n_{k}}}{\widetilde{N}}\right)-\theta_{k}\log \theta_{k}\\
& =\frac{2^{n_{k}}}{\widetilde{N}}\log\left(\frac{2^{n_{k}}}{\widetilde{N}}\right)\left(\frac{\widetilde{N}}{N}-1\right)+\frac{2^{n_{k}}}{\widetilde{N}}\frac{\widetilde{N}}{N}\log\left(\frac{\widetilde{N}}{N}\right)\\
& +\frac{2^{n_{k}}}{\widetilde{N}}\log\left(\frac{2^{n_{k}}}{\widetilde{N}}\right)-\theta_{k}\log \theta_{k}.
\end{align*}
From \eqref{subgentheta} and \eqref{closeNN1} we obtain
\[
\limsup_{N\in\widehat{\mathcal{N}}}\left|\frac{2^{n_{k}}}{N}\log\left(\frac{2^{n_{k}}}{N}\right)-\theta_{k}\log \theta_{k}\right|\leq |\theta_{k}\log \theta_{k}|\,\epsilon+\theta_{k}|\log(1-\epsilon)|
\] 
which implies
\begin{align}
\limsup_{N\in\widehat{\mathcal{N}}}\left|\sum_{k=1}^{M}\frac{2^{n_{k}}}{N}\log\left(\frac{2^{n_{k}}}{N}\right)-\Lambda(\vec{\theta})\right| & =\limsup_{N\in\widehat{\mathcal{N}}}\left|\sum_{k=1}^{M}\left(\frac{2^{n_{k}}}{N}\log\left(\frac{2^{n_{k}}}{N}\right)-\theta_{k} \log \theta_{k}\right)\right|\notag\\
& \leq \sum_{k=1}^{M}\limsup_{N\in\widehat{\mathcal{N}}}\left|\frac{2^{n_{k}}}{N}\log\left(\frac{2^{n_{k}}}{N}\right)-\theta_{k}\log \theta_{k}\right|\notag\\
& \leq |\Lambda(\vec{\theta})|\,\epsilon+|\log(1-\epsilon)|\notag\\
& \leq |\lambda|\,\epsilon+|\log(1-\epsilon)|.\label{estsecpart}
\end{align} 
From the definition of $S_{N,1}$ and the estimates \eqref{estfirstpart} and \eqref{estsecpart}, we easily deduce that
\begin{equation}\label{estimateforSN1}
\limsup_{N\in\widehat{\mathcal{N}}}|S_{N,1}-(L+\frac{1}{\pi} \Lambda(\vec{\theta}))|\leq L\,\epsilon+\frac{1}{\pi}(|\lambda|\,\epsilon+|\log(1-\epsilon)|).
\end{equation}

If we write
\begin{align*}
\frac{U_{N,1}(a_N)-\frac{1}{\pi}N\log N}{N} & =S_{N,1}+S_{N,2}-(L+\frac{1}{\pi}\Lambda(\vec{\theta}))+L+\frac{1}{\pi}\Lambda(\vec{\theta})\\
& \leq |S_{N,2}|+|S_{N,1}-(L+\frac{1}{\pi}\Lambda(\vec{\theta}))|+L
\end{align*}
where we used that $\Lambda(\vec{\theta})\leq 0$, we obtain from \eqref{aboundforSN2} and \eqref{estimateforSN1} that 
\begin{align*}
\lim_{N\in\mathcal{N}}\frac{U_{N,1}(a_N)-\frac{1}{\pi}N\log N}{N} & =\lim_{N\in\widehat{\mathcal{N}}}\frac{U_{N,1}(a_N)-\frac{1}{\pi}N\log N}{N}\\
& \leq \limsup_{N\in\widehat{\mathcal{N}}}\,\,(|S_{N,2}|+|S_{N,1}-(L+\frac{1}{\pi}\Lambda(\vec{\theta}))|+L)\\
& \leq (C+\frac{1}{\pi}\log 2)\,\epsilon+L\,\epsilon+\frac{1}{\pi}(|\lambda|\,\epsilon+|\log(1-\epsilon)|)+L,
\end{align*}
and letting $\epsilon\rightarrow 0$ we reach the desired inequality
\begin{equation}\label{desineq}
\lim_{N\in\mathcal{N}}\frac{U_{N,1}(a_N)-\frac{1}{\pi}N\log N}{N}\leq L.
\end{equation}
This finishes the proof of \eqref{eq:secordlimsups1}. 

In view of \eqref{ineqsecordliminfs1}, the proof of \eqref{eq:secordliminfs1} will be complete if we show
\[
\liminf_{N\rightarrow\infty}\frac{U_{N,1}(a_N)-\frac{1}{\pi}N\log N}{N}\geq \frac{1}{\pi}\left(\gamma + \log\left(\frac{8}{\pi}\right)+\lambda\right).
\] 
Let $\mathcal{N}$ be a sequence of natural numbers for which \eqref{subseqconv} holds, and let us prove that the limit is $\geq (1/\pi) (\gamma + \log(8/\pi)+\lambda)$. As argued before and in virtue of \eqref{impcons}, the claim is true if there exists $p\in\mathbb{N}$ such that $\tau_{b}(N)=p$ for infinitely many $N\in\mathcal{N}$. Now assume that $\tau_{b}(N)\rightarrow\infty$ along the sequence $\mathcal{N}$. We can use the same estimates obtained before in the proof of \eqref{desineq}. Indeed, for $\vec{\theta}=(\theta_{1},\ldots,\theta_{M})$ the vector given by \eqref{subgentheta}, writing
\begin{align*}
\frac{U_{N,1}(a_{N})-\frac{1}{\pi}N \log N}{N} & =S_{N,1}+S_{N,2}-(L+\frac{1}{\pi}\Lambda(\vec{\theta}))+L+\frac{1}{\pi}\Lambda(\vec{\theta})\\
& \geq \frac{1}{\pi}(\gamma+\log\left(\frac{8}{\pi}\right)+\lambda)-|S_{N,2}|-|S_{N,1}-(L+\frac{1}{\pi}\Lambda(\vec{\theta}))|
\end{align*}
and applying \eqref{aboundforSN2} and \eqref{estimateforSN1} we obtain
\[
\lim_{N\in\mathcal{N}}\frac{U_{N,1}(a_{N})-\frac{1}{\pi}N \log N}{N}\geq \frac{1}{\pi}(\gamma+\log\left(\frac{8}{\pi}\right)+\lambda)-(C+L+\frac{1}{\pi}(\log 2+|\lambda|))\,\epsilon-\frac{1}{\pi}|\log(1-\epsilon)|.
\]
Letting $\epsilon\rightarrow 0$ we obtain the desired inequality. This concludes the proof of \eqref{eq:secordliminfs1}.\qed

\begin{remark}
Let $N(p)=2^{p}-1$. As in \eqref{eq:UN(p)}, we have $U_{N(p),1}(a_{N(p)})=2^{-p}\mathcal{L}_{1}(2^p)$. Hence by \eqref{limauxBHS} we obtain 
\begin{align*}
\lim_{p\rightarrow\infty}\frac{U_{N(p),1}(a_{N(p)})-\frac{1}{\pi}N(p)\log(N(p))}{N(p)}& =\lim_{p\rightarrow\infty}\frac{\mathcal{L}_{1}(2^{p})-\frac{1}{\pi} 2^{p}(2^p-1)\log(2^{p}-1)}{2^{p}(2^p-1)}\\
& =\frac{1}{\pi}(\gamma+\log(2/\pi)).
\end{align*}
Comparing this limit value with the limit value in \eqref{eq:secordliminfs1}, we deduce the estimate
\[
\lambda\leq -2\log 2.
\]
\end{remark}

\section{The case $s>1$}\label{sec:sgone}

We need first the following auxiliary result.

\begin{proposition}
If $s > 1$, then 
\begin{equation}\label{asympUsncalsg1}
\lim_{N \to \infty} \frac{\mathcal{U}_{s}(N)}{N^{s}} = (2^s - 1)\frac{2\zeta(s)}{(2\pi)^s}.
\end{equation}
\end{proposition}
\begin{proof}
It follows from \cite[Thm. 3.2]{MMRS} that if $s > 1$, then
\begin{equation}\label{asympLsnsg1}
\lim_{N \to \infty} \frac{\mathcal{L}_{s}(N)}{N^{1+s}} = \frac{2\zeta(s)}{(2\pi)^s}.
\end{equation}
Applying \eqref{calUsNLsN} we obtain
\begin{align*}
\lim_{N \to \infty} \frac{\mathcal{U}_{s}(N)}{N^s} &= \lim_{N \to \infty} \frac{1}{N^s}\left(\frac{\mathcal{L}_{s}(2N)}{2N} - \frac{\mathcal{L}_{s}(N)}{N}\right) \\
&= \lim_{N \to \infty} \left(2^s \frac{\mathcal{L}_{s}(2N)}{(2N)^{1+s}} - \frac{\mathcal{L}_{s}(N)}{N^{1+s}}\right) \\
&= (2^s - 1)\frac{2\zeta(s)}{(2\pi)^s}.
\end{align*}
\end{proof}

In this section we use the notation
\[
\mathcal{W}_{s}(N):=\frac{\mathcal{U}_{s}(N)}{N^{s}},\qquad s>1.
\]
Therefore we have
\begin{equation}\label{asympWsNsg1}
\lim_{N\rightarrow\infty}\mathcal{W}_{s}(N)=(2^{s}-1)\frac{2\zeta(s)}{(2\pi)^{s}},\qquad s>1.
\end{equation}
Note that this limit value is now positive, contrary to the situation in the range $0<s<1$.

Since we are assuming that $s>1$, for the function $G$ defined in \eqref{def:Gcapfunc} we have
\begin{equation}\label{eq:boundsGsg1}
0<G((\theta_{1},\ldots,\theta_{p});s)=\sum_{k=1}^{p}\theta_{k}^{s} \leq\sum_{k=1}^{p}\theta_{k}=1.
\end{equation}
Recall that we defined $\underline{g}(s):=\inf_{\vec{\theta}\in\Theta}G(\vec{\theta};s)$.

\begin{figure*}
\centering
\begin{subfigure}[b]{0.475\textwidth}
\centering
\includegraphics[width=\textwidth]{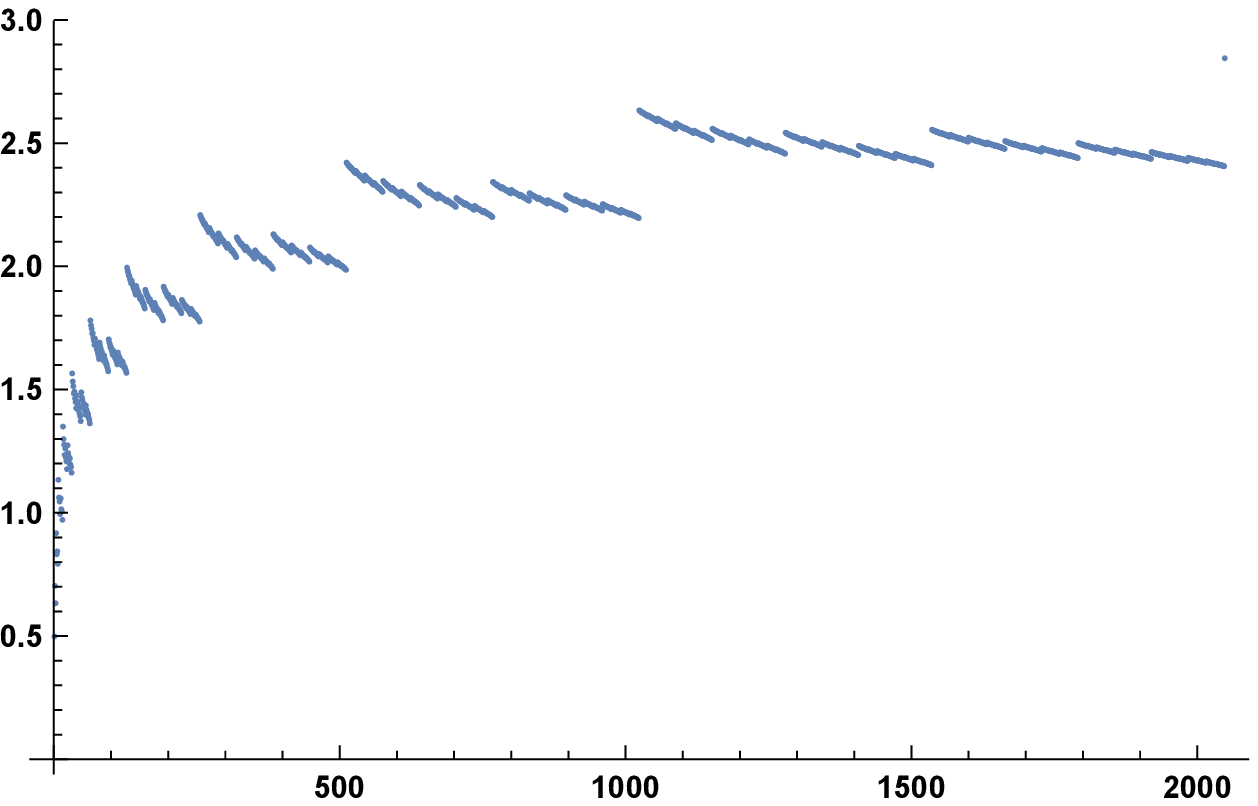}
\caption[]{{$s=1.005$}}    
\end{subfigure}
\hfill
\begin{subfigure}[b]{0.475\textwidth}
\centering 
\includegraphics[width=\textwidth]{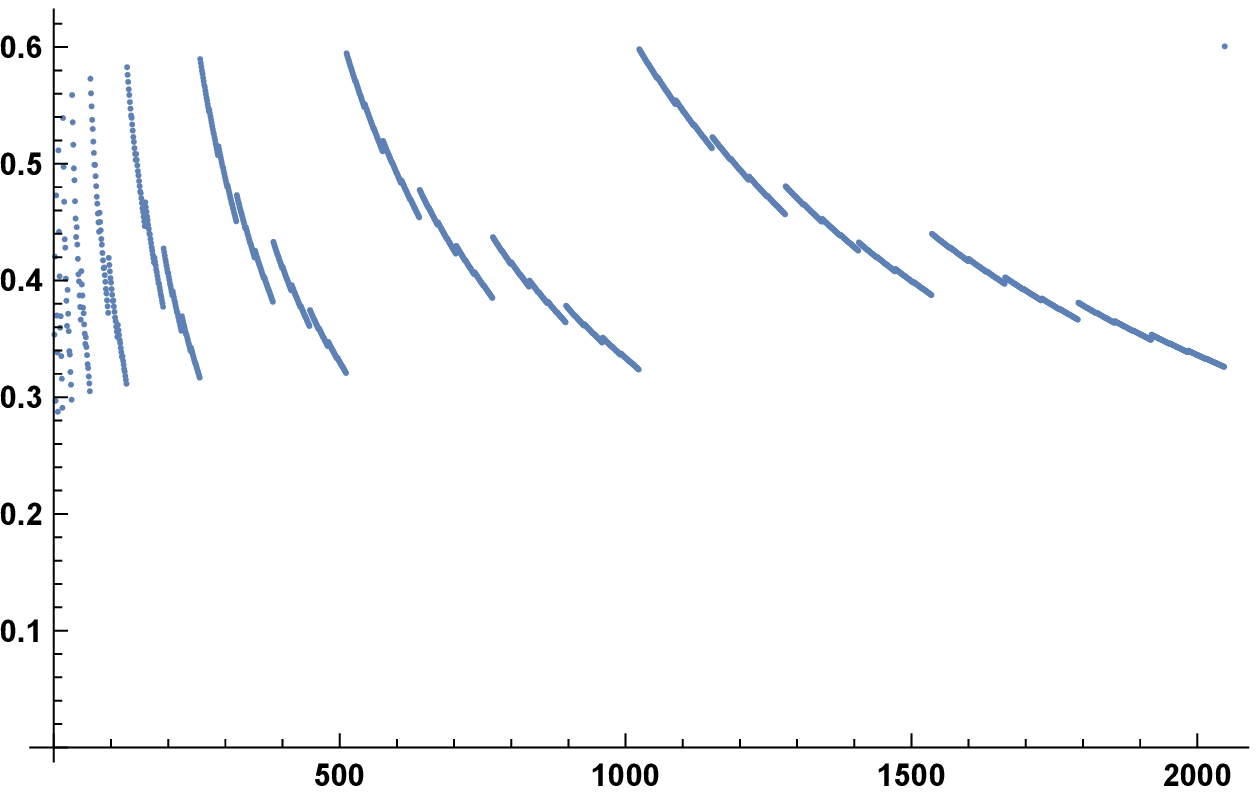}
\caption[]{{$s=1.5$}}
\end{subfigure}
\vskip\baselineskip
\begin{subfigure}[b]{0.475\textwidth}
\centering
\includegraphics[width=\textwidth]{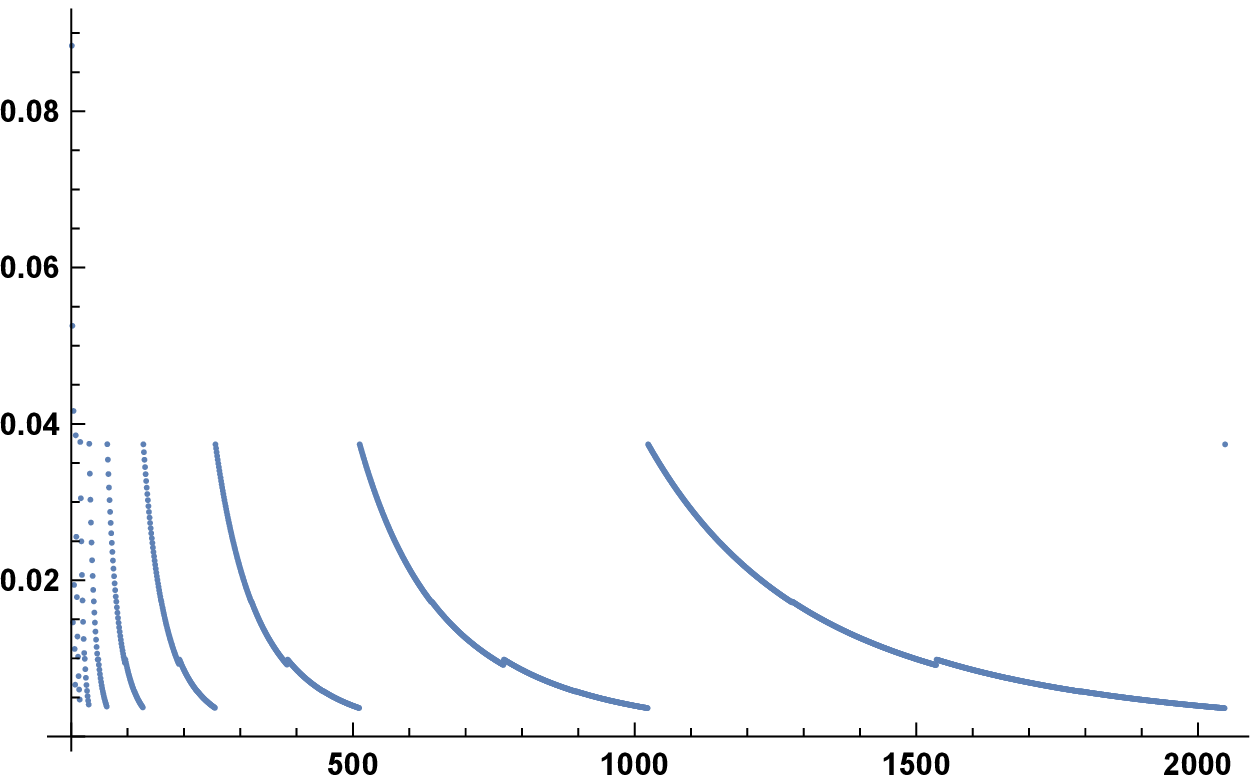}
\caption[]{{$s=3.5$}}    
\end{subfigure}
\hfill
\begin{subfigure}[b]{0.475\textwidth}
\centering 
\includegraphics[width=\textwidth]{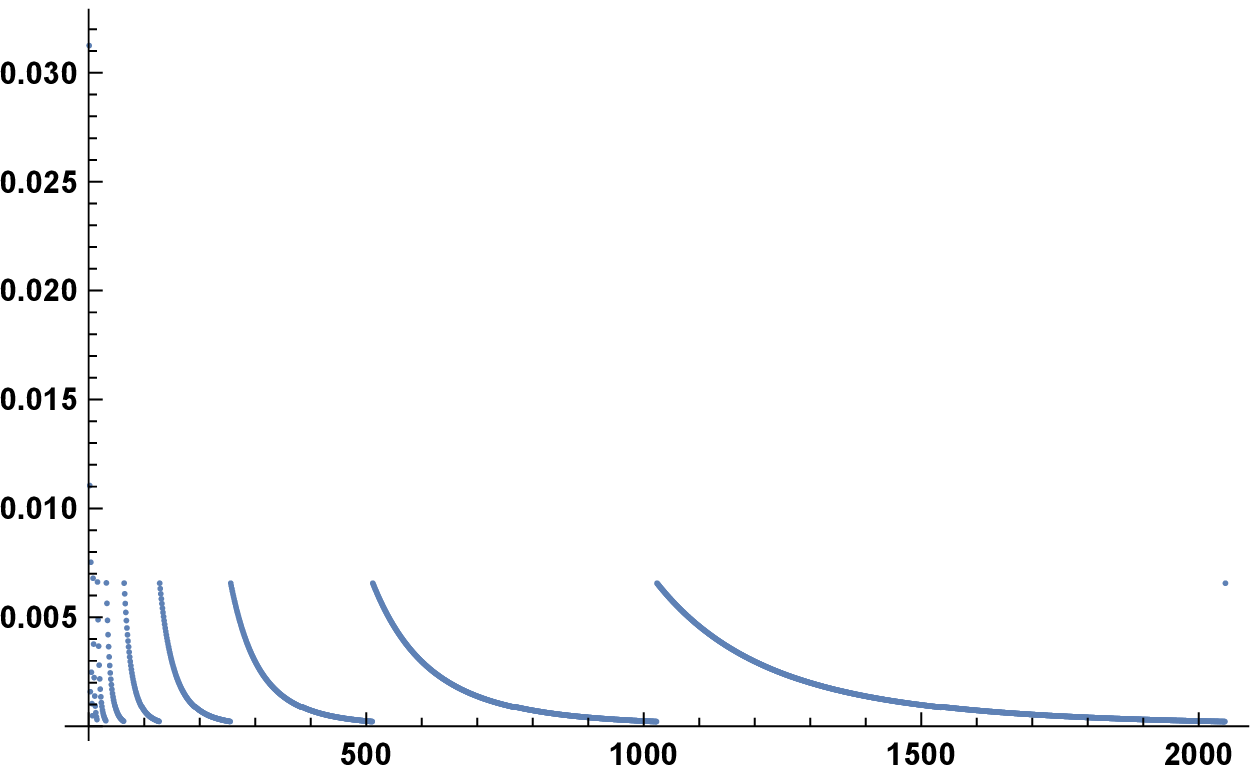}
\caption[]{{$s=5$}}
\end{subfigure}
\caption[]{{Plots of sequences $\Big(\frac{U_{N,s}(a_{N})}{N^{s}}\Big)$, $1\leq N\leq 2048$, for different values of $s>1$.}}
\label{fig4}
\end{figure*}

\noindent\textit{Proof of Theorem}~\ref{theo:foasg1}. If $N=2^{n_{1}}+\cdots+2^{n_{p}}$, then
\[
\frac{U_{N,s}(a_{N})}{N^{s}}=\sum_{k=1}^{p}\frac{\mathcal{U}_{s}(2^{n_{k}})}{N^{s}}
=\sum_{k=1}^{p}\mathcal{W}_{s}(2^{n_{k}})\left(\frac{2^{n_{k}}}{N}\right)^{s}.
\]
If $C>0$ is an upper bound for the sequence $(\mathcal{W}_{s}(N))$, then
\[
0\leq \frac{U_{N,s}(a_{N})}{N^{s}}=\sum_{k=1}^{p}\mathcal{W}_{s}(2^{n_{k}})\left(\frac{2^{n_{k}}}{N}\right)^{s}\leq C\sum_{k=1}^{p}\left(\frac{2^{n_{k}}}{N}\right)^{s}\leq C\sum_{k=1}^{p}\frac{2^{n_{k}}}{N}=C
\] 
so \eqref{seqUnsansg1} is bounded.

To prove \eqref{asymplimsupUsg1}--\eqref{asympliminfUsg1} we can use the same argument in the proof of \eqref{eq:limsupunan}--\eqref{eq:liminfunan}, with some minor changes which we indicate below. Taking the subsequence $(2^{n})$ and applying \eqref{asympWsNsg1} we get
\[
\lim_{n\rightarrow\infty}\frac{U_{2^{n},s}(a_{2^{n}})}{2^{ns}}=\lim_{n\rightarrow\infty}\mathcal{W}_{s}(2^{n})=(2^{s}-1)\frac{2\zeta(s)}{(2\pi)^{s}},
\]  
which implies
\[
\limsup_{N\rightarrow\infty}\frac{U_{N,s}(a_{N})}{N^{s}}\geq (2^{s}-1)\frac{2\zeta(s)}{(2\pi)^{s}}.
\]

If $\vec{\theta}=(\theta_{1},\ldots,\theta_{p})\in\Theta_{p}$ is arbitrary, and $\mathcal{N}$ is a sequence of integers $N=2^{n_{1}}+\cdots+2^{n_{p}}$ such that \eqref{seqassoctheta} holds, then as in \eqref{limspesub} we obtain
\begin{equation}\label{eqineq}
\lim_{N\in\mathcal{N}}\frac{U_{N,s}(a_{N})}{N^{s}}=(2^{s}-1)\frac{2\zeta(s)}{(2\pi)^{s}} G(\vec{\theta};s)\leq (2^{s}-1)\frac{2\zeta(s)}{(2\pi)^{s}},
\end{equation}
where we used \eqref{eq:boundsGsg1}.

To complete the proof of \eqref{asymplimsupUsg1}, we consider as in \eqref{eq:convsubseq} a sequence $\mathcal{N}$ such that
\begin{equation}\label{convsubsg1}
\left(\frac{U_{N,s}(a_{N})}{N^{s}}\right)_{N\in\mathcal{N}}
\end{equation}
converges, and we show that
\begin{equation}\label{desineqlimsup}
\lim_{N\in\mathcal{N}}\frac{U_{N,s}(a_{N})}{N^{s}}\leq (2^{s}-1)\frac{2\zeta(s)}{(2\pi)^{s}}.
\end{equation}
If there exists $p\in\mathbb{N}$ such that the $\tau_{b}(N)=p$ for infinitely many $N\in\mathcal{N}$, then \eqref{eqineq} holds.

If $\tau_{b}(N)\rightarrow\infty$ along $\mathcal{N}$, we do the decomposition 
\begin{equation}\label{decompUns}
\frac{U_{N,s}(a_{N})}{N^{s}}=S_{N,1}+S_{N,2}
\end{equation}
with $S_{N,1}$ and $S_{N,2}$ defined as in \eqref{def:SN1sl1}--\eqref{def:SN2sl1} (with our current definition of $\mathcal{W}_{s}(N)$), corresponding to some choice of $0<\epsilon<1$ and $M$ such that \eqref{smalltail} holds. Then we have the bound \eqref{est:SN2}. Using the inequality
\[
S_{N,1}=\sum_{k=1}^{M}\mathcal{W}_{s}(2^{n_{k}})\left(\frac{2^{n_{k}}}{\widetilde{N}}\right)^{s}\left(\frac{\widetilde{N}}{N}\right)^{s}\leq \sum_{k=1}^{M}\mathcal{W}_{s}(2^{n_{k}})\left(\frac{2^{n_{k}}}{\widetilde{N}}\right)^{s},
\]
where $\widetilde{N}=2^{n_{1}}+2^{n_{2}}+\cdots+2^{n_{M}}$, we obtain, passing to a subsequence $\widehat{\mathcal{N}}\subset\mathcal{N}$ satisfying \eqref{def:subseqNhat}, the inequality
\[
\lim_{N\in\mathcal{N}}\frac{U_{N,s}(a_{N})}{N^{s}}
=\limsup_{N\in\widehat{\mathcal{N}}}(S_{N,1}+S_{N,2})\leq (2^{s}-1)\frac{2\zeta(s)}{(2\pi)^{s}}+C\epsilon,
\]
and letting $\epsilon\rightarrow 0$ we obtain \eqref{desineqlimsup}.

The equality in \eqref{eqineq}, valid for every $\vec{\theta}\in\Theta_{p}$, $p\geq 1$,  immediately implies
\[
\liminf_{N\rightarrow\infty}\frac{U_{N,s}(a_{N})}{N^{s}}\leq \underline{g}(s)(2^{s}-1)\frac{2\zeta(s)}{(2\pi)^{s}}.
\] 
To prove the reverse inequality
\[
\liminf_{N\rightarrow\infty}\frac{U_{N,s}(a_{N})}{N^{s}}\geq \underline{g}(s)(2^{s}-1)\frac{2\zeta(s)}{(2\pi)^{s}},
\]
we take as before a sequence $\mathcal{N}$ such that \eqref{convsubsg1} converges, and show that
\begin{equation}\label{revineqliminf}
\lim_{N\in\mathcal{N}}\frac{U_{N,s}(a_{N})}{N^{s}}\geq \underline{g}(s)(2^{s}-1)\frac{2\zeta(s)}{(2\pi)^{s}}.
\end{equation}
We use again the decomposition \eqref{decompUns}. In the case that $\tau_{b}(N)\rightarrow\infty$ along the sequence $\mathcal{N}$, we need to use the inequality 
\[
S_{N,1}=\left(\frac{\widetilde{N}}{N}\right)^{s}\sum_{k=1}^{M}\mathcal{W}_{s}(2^{n_{k}})\left(\frac{2^{n_{k}}}{\widetilde{N}}\right)^{s}\geq (1-\epsilon)^{s}\sum_{k=1}^{M}\mathcal{W}_{s}(2^{n_{k}})\left(\frac{2^{n_{k}}}{\widetilde{N}}\right)^{s}
\]
and passing to an appropriate subsequence $\widehat{\mathcal{N}}\subset\mathcal{N}$ satisfying \eqref{def:subseqNhat}, using $S_{N,2}\geq 0$, we will get
\begin{align*}
\lim_{N\in\mathcal{N}} \frac{U_{N,s}(a_{N})}{N^{s}} & =\lim_{N\in\widehat{\mathcal{N}}}(S_{N,1}+S_{N,2})\geq \liminf_{N\in\widehat{\mathcal{N}}} S_{N,1}\\
& \geq (1-\epsilon)^{s} G(\vec{\theta};s) (2^{s}-1)\frac{2\zeta(s)}{(2\pi)^{s}}\\
& \geq (1-\epsilon)^{s}\underline{g}(s) (2^{s}-1)\frac{2\zeta(s)}{(2\pi)^{s}}
\end{align*}
for some $\vec{\theta}\in\Theta_{M}$, which implies \eqref{revineqliminf} after taking $\epsilon\rightarrow 0$.\qed

\begin{remark}
As in Proposition~\ref{prop:limPowerOf2Minus1}, using \eqref{asympLsnsg1} we can show that for the sequence $N(p)=2^{p}-1$ we have
\[
\lim_{p\rightarrow\infty}\frac{U_{N(p),s}(a_{N(p)})}{N(p)^{s}}=\frac{2\zeta(s)}{(2\pi)^{s}},\qquad s>1.
\]
Hence
\[
\liminf_{N\rightarrow\infty}\frac{U_{N,s}(a_{N})}{N^s}\leq \frac{2\zeta(s)}{(2\pi)^{s}}.
\]
This and \eqref{asympliminfUsg1} imply in particular that
\[
\underline{g}(s)\leq \frac{1}{2^{s}-1},\qquad s>1.
\]
\end{remark}

\section{Proof of Theorem~\ref{theo:genseq}}\label{sec:genseq}

The method we employ is the same one used in the case $p=0$ (as illustrated for example in \cite[Section V.1]{SaffTotik}), only slightly adapted to the case $p\geq 1$. Let $0\leq s<1$ and $k_{s}(z,w)$ be the kernel defined in \eqref{def:kernelks}. We shall use the notation
\[
U_{s}^{\sigma}(z)=\int_{S^{1}}k_{s}(z,w)\,d\sigma(w).
\] 
Clearly, $U_{s}^{\sigma}$ is constant on $S^{1}$ and its value is $I_{s}(\sigma)$, which equals zero if $s=0$.  

Let $(a_{n})_{n=0}^{\infty}$ be a greedy $s$-energy sequence on $S^{1}$ with initial set $(a_0,\ldots,a_{p})$ of $p+1$ distinct points, $p\geq 1$. For every $N\geq p+2$, we have
\begin{align*}
E_{s}(\alpha_{N,s}) & =2\sum_{j=1}^{N-1}\sum_{i=0}^{j-1}k_{s}(a_{j},a_{i})=2\sum_{j=1}^{p}\sum_{i=0}^{j-1}k_{s}(a_{j},a_{i})+2\sum_{j=p+1}^{N-1}\sum_{i=0}^{j-1}k_{s}(a_{j},a_{i})\\
& \leq E_{s}(\alpha_{p+1,s})+2\sum_{j=p+1}^{N-1}\sum_{i=0}^{j-1}k_{s}(z,a_{i})
\end{align*}
valid for all $z\in S^{1}$, where we have applied \eqref{eq:optgencase}. Integrating this inequality with respect to the probability measure $\sigma$, we get
\begin{align}
E_{s}(\alpha_{N,s}) & \leq E_{s}(\alpha_{p+1,s})+2\sum_{j=p+1}^{N-1}\sum_{i=0}^{j-1} U_{s}^{\sigma}(a_{i})\notag\\
& =E_{s}(\alpha_{p+1,s})+I_{s}(\sigma)(N(N-1)-p(p+1)).\label{eq:keyest}
\end{align}
On the other hand, if $\omega_{N}$ denotes the configuration formed by the $N$-th roots of unity, we have $E_{s}(\omega_{N})\leq E_{s}(\alpha_{N,s})$ for all $0\leq s<1$ and $N\geq 2$, and we also have 
\begin{equation}\label{eq:asympeqspaced}
\lim_{N\rightarrow\infty}\frac{E_{s}(\omega_{N})}{N^{2}}=I_{s}(\sigma)
\end{equation}
(cf. Theorems 2.3.3 and 4.4.9 in \cite{BorHarSaff}). So the inequality $E_{s}(\omega_{N})\leq E_{s}(\alpha_{N,s})$ combined with \eqref{eq:keyest} and \eqref{eq:asympeqspaced} imply 
\begin{equation}\label{eq:asympenergbisbis}
\lim_{N\rightarrow\infty}\frac{E_{s}(\alpha_{N,s})}{N^{2}}=I_{s}(\sigma).
\end{equation} 
In virtue of Theorem 4.4.9 in \cite{BorHarSaff}, \eqref{eq:asympenergbisbis} in turn implies \eqref{eq:unifdist}.

Now we assume $s>0$ and justify the formula
\begin{equation}\label{asympformU1}
\lim_{n\rightarrow\infty}\frac{U_{N,s}(a_{N})}{N}=I_{s}(\sigma),\qquad 0<s<1.
\end{equation}
If $n\geq p+1$,
\begin{equation}\label{ineq1}
U_{n,s}(a_{n})=\sum_{k=0}^{n-1}\frac{1}{|a_{n}-a_{k}|^{s}}\leq \sum_{k=0}^{n-1}\frac{1}{|z-a_{k}|^{s}}\qquad\mbox{for all}\,\,z\in S^{1}.
\end{equation}
If we take $z=a_{n+1}$ in this inequality we get
\begin{equation}\label{ineq1bis}
U_{n,s}(a_{n})\leq \sum_{k=0}^{n-1}\frac{1}{|a_{n+1}-a_{k}|^{s}}=U_{n+1,s}(a_{n+1})-\frac{1}{|a_{n+1}-a_{n}|^{s}}
\end{equation}
hence
\begin{equation}\label{ineq2}
U_{n,s}(a_{n})\leq U_{n+1,s}(a_{n+1}),\qquad n\geq p+1.
\end{equation}
Integrating both sides of the inequality \eqref{ineq1} with respect to $\sigma$ we also obtain
\begin{equation}\label{ineq3}
U_{n,s}(a_{n})\leq n\,I_{s}(\sigma),\qquad n\geq p+1.
\end{equation}

\noindent\textbf{Claim:} Let $0<\epsilon<I_{s}(\sigma)$ and suppose that $m\geq p+2$ is such that
\begin{equation}\label{ineq4}
\frac{U_{m,s}(a_{m})}{m}<I_{s}(\sigma)-\epsilon.
\end{equation}
Then, for any $p+1\leq \kappa_{m}<m$, we have
\begin{equation}\label{ineq5}
E_{s}(\alpha_{m+1,s})< E_{s}(\alpha_{p+1,s})+I_{s}(\sigma)(\kappa_{m}(\kappa_{m}+1)-p(p+1))+2m(m-\kappa_{m})(I_{s}(\sigma)-\epsilon).
\end{equation}
Indeed, we have
\begin{align*}
\frac{1}{2}E_{s}(\alpha_{m+1,s}) & =\sum_{i=1}^{m}U_{i,s}(a_{i})=\sum_{i=1}^{p}U_{i,s}(a_{i})+\sum_{i=p+1}^{\kappa_{m}}U_{i,s}(a_{i})+\sum_{i=\kappa_{m}+1}^{m}U_{i,s}(a_{i})\\
& <\frac{1}{2}E_{s}(\alpha_{p+1,s})+\sum_{i=p+1}^{\kappa_{m}}i I_{s}(\sigma)+\sum_{i=\kappa_{m}+1}^{m}m(I_{s}(\sigma)-\epsilon)
\end{align*}
where we have used \eqref{ineq3}, \eqref{ineq2}, and $U_{i,s}(a_{i})\leq U_{m,s}(a_{m})<m(I_{s}(\sigma)-\epsilon)$ for all $i\leq m$. So \eqref{ineq5} is justified.

Now we finish the proof of \eqref{asympformU1}. Let $0<\epsilon<I_{s}(\sigma)$. In view of \eqref{ineq3}, it suffices to show that only finitely many indices $m$ satisfy \eqref{ineq4}. So assume that for infinitely many $m\geq p+2$ we have \eqref{ineq4}. Let $\mathcal{N}$ be the sequence formed by such indices $m$. Then from \eqref{ineq5} we obtain
\begin{equation}\label{ineq6}
\frac{E_{s}(\alpha_{m+1,s})}{(m+1)^2}\leq I_{s}(\sigma)\frac{\kappa_{m}(\kappa_{m}+1)}{(m+1)^2}+2(I_{s}(\sigma)-\epsilon)\frac{m(m-\kappa_{m})}{(m+1)^2}+\eta_{m}
\end{equation}
where $\eta_{m}=O(m^{-2})$. Now we choose $\kappa_{m}$ so that
\begin{equation}\label{limnormkm}
\lim_{m\in\mathcal{N}}\frac{\kappa_{m}}{m}=1-\frac{\epsilon}{I_{s}(\sigma)}.
\end{equation}
This is possible since $0<1-\epsilon/I_{s}(\sigma)<1$. If we call $l_{\epsilon}$ the limit value in \eqref{limnormkm}, then from \eqref{ineq6} we deduce that
\[
\limsup_{m\in\mathcal{N}} \frac{E_{s}(\alpha_{m+1,s})}{(m+1)^2}\leq I_{s}(\sigma) l_{\epsilon}^2+2(I_{s}(\sigma)-\epsilon)(1-l_{\epsilon})=I_{s}(\sigma)-\frac{\epsilon^2}{I_{s}(\sigma)}<I_{s}(\sigma)
\] 
which contradicts \eqref{eq:asympenergbisbis}. Therefore only finitely many $m$'s satisfy \eqref{ineq4} and this concludes the proof of \eqref{asympformU1}. 

Finally we assume $s=0$ and prove 
\begin{equation}\label{asympformU2}
\lim_{N\rightarrow\infty}\frac{U_{N,0}(a_{N})}{N}=0.
\end{equation}
Arguing as in \eqref{ineq1}, \eqref{ineq1bis}, and \eqref{ineq3}, we obtain
\[
U_{n,0}(a_{n})\leq U_{n+1,0}(a_{n+1})-\log\frac{1}{|a_{n+1}-a_{n}|},\qquad n\geq p+1,
\]
and
\begin{equation}\label{ineq7}
U_{n,0}(a_{n})\leq n\,I_{0}(\sigma)=0,\qquad n\geq p+1.
\end{equation}
Hence
\begin{equation}\label{ineq8}
U_{n,0}(a_{n})\leq U_{n+1,0}(a_{n+1})+L,\qquad n\geq p+1,
\end{equation}
where $L=\log 2$.

Let $0<\epsilon<L$, and suppose that there are infinitely many $m\geq p+2$ such that
\begin{equation}\label{ineq9}
\frac{U_{m,0}(a_{m})}{m}<-\epsilon.
\end{equation}
We will reach a contradiction, which together with \eqref{ineq7} will imply \eqref{asympformU2}. Let $\mathcal{N}$ be the sequence formed by all $m$ satisfying \eqref{ineq9}. A repeated application of \eqref{ineq8} shows that for every $p+1\leq i\leq m$, $m\in\mathcal{N}$, we have
\[
U_{i,0}(a_{i})\leq U_{m,0}(a_{m})+(m-i)L<-\epsilon\, m+(m-i)L.
\]
Using this estimate and \eqref{ineq7}, we deduce that for any $p+1\leq \kappa_{m}<m$, 
\begin{align*}
\frac{1}{2}E_{0}(\alpha_{m+1,0}) & =\sum_{i=1}^{m}U_{i,0}(a_{i})=\sum_{i=1}^{p}U_{i,0}(a_{i})+\sum_{i=p+1}^{\kappa_{m}} U_{i,0}(a_{i})+\sum_{i=\kappa_{m}+1}^{m}U_{i,0}(a_{i})\\
& <\frac{1}{2} E_{0}(\alpha_{p+1,0})+\sum_{i=\kappa_{m}+1}^{m}(-\epsilon\,m+(m-i)L)\\
& =\frac{1}{2}E_{0}(\alpha_{p+1,0})-m\,\epsilon\,(m-\kappa_{m})+\frac{L}{2}(m-\kappa_{m}-1)(m-\kappa_{m}).
\end{align*}
Hence
\begin{equation}\label{ineq10}
\frac{E_{0}(\alpha_{m+1,0})}{(m+1)^{2}}\leq -\frac{2m\epsilon(m-\kappa_{m})}{(m+1)^2}+L\frac{(m-\kappa_{m}-1)(m-\kappa_{m})}{(m+1)^2}+\eta_{m}
\end{equation}
where $\eta_{m}=O(m^{-2})$. We can choose $\kappa_{m}$ so that
\[
\lim_{m\in\mathcal{N}}\frac{\kappa_{m}}{m}=1-\frac{\epsilon}{L}\in(0,1).
\]
Let $l_{\epsilon}=1-\epsilon/L$. Taking limit in \eqref{ineq10} we get
\[
\limsup_{m\in\mathcal{N}}\frac{E_{0}(\alpha_{m+1,0})}{(m+1)^2}\leq -2\epsilon (1-l_{\epsilon})+L(1-l_{\epsilon})^2=-\frac{\epsilon^2}{L}<0.
\]
This is a contradiction since we have shown that $\lim_{N\rightarrow\infty} E_{0}(\alpha_{N,0})/N^2=0$.

\bigskip

\section{Statements}

The authors declare that there are no conflicts of interest. This research did not receive any specific grant from funding agencies in the public, commercial, or not-for-profit sectors.

\bigskip

\noindent \textsc{Department of Mathematics, University of Central Florida, 4393 Andromeda Loop North, Orlando, FL 32816} \\
\textit{Email address}: \texttt{abey.lopez-garcia\symbol{'100}ucf.edu}

\bigskip

\noindent \textsc{Department of Mathematics, University of Central Florida, 4393 Andromeda Loop North, Orlando, FL 32816}\\
\textit{Email address}: \texttt{remccleary\symbol{'100}knights.ucf.edu}

\end{document}